\documentclass[10pt]{amsart}
\usepackage{amssymb}
\usepackage{bm}
\usepackage{graphicx}
\usepackage[centertags]{amsmath}
\usepackage{amsfonts}
\usepackage{amsthm}
\newtheorem{thm}{Theorem}
\newtheorem{cor}[thm]{Corollary}
\newtheorem{lem}[thm]{Lemma}

\newtheorem{prop}[thm]{Proposition}
\newtheorem{claim}[thm]{Claim}
\newtheorem{fact}[thm]{Fact}

\newtheorem{defn}[thm]{Definition}
\theoremstyle{definition}

\newcommand{\rr}{\mathbb{R}}
\newcommand{\nn}{\mathbb{N}}
\newcommand{\ee}{\varepsilon}

\newcommand{\suc}{\mathrm{Succ}}
\newcommand{\immsuc}{\mathrm{ImmSucc}}
\newcommand{\dens}{\mathrm{dens}}
\newcommand{\fan}{\mathrm{Fan}}
\newcommand{\dhl}{\mathrm{DHL}}
\newcommand{\hl}{\mathrm{HL}}
\newcommand{\strong}{\mathrm{strong}}
\newcommand{\ci}{\mathrm{I}}
\newcommand{\mhden}{\mathbf{0}}

\newcommand{\bfp}{\mathbf{p}}
\newcommand{\bfr}{\mathbf{r}}
\newcommand{\bfs}{\mathbf{s}}
\newcommand{\bft}{\mathbf{t}}
\newcommand{\bfu}{\mathbf{u}}
\newcommand{\bfv}{\mathbf{v}}

\newcommand{\bfz}{\mathbf{z}}
\newcommand{\bfcb}{\mathbf{B}}
\newcommand{\bfci}{\mathbf{I}}
\newcommand{\bfcf}{\mathbf{F}}
\newcommand{\bfcr}{\mathbf{R}}
\newcommand{\bfcs}{\mathbf{S}}
\newcommand{\bfct}{\mathbf{T}}
\newcommand{\bfcu}{\mathbf{U}}
\newcommand{\bfcv}{\mathbf{V}}

\newcommand{\bfcz}{\mathbf{Z}}
\newcommand{\meg}{\geqslant}
\newcommand{\mik}{\leqslant}
\newcommand{\lex}{<_{\mathrm{lex}}}
\newcommand{\con}{\smallfrown}

\begin{document}

\title{A density version of the Halpern--L\"{a}uchli theorem}

\author{Pandelis Dodos, Vassilis Kanellopoulos and Nikolaos Karagiannis}

\address{Department of Mathematics, University of Athens, Panepistimiopolis 157 84, Athens, Greece}
\email{pdodos@math.uoa.gr}

\address{National Technical University of Athens, Faculty of Applied Sciences,
Department of Mathematics, Zografou Campus, 157 80, Athens, Greece}
\email{bkanel@math.ntua.gr}

\address{National Technical University of Athens, Faculty of Applied Sciences,
Department of Mathematics, Zografou Campus, 157 80, Athens, Greece}
\email{nkaragiannis@math.ntua.gr}


\maketitle


\begin{abstract}
We prove a density version of the Halpern--L\"{a}uchli Theorem. This settles in the affirmative a conjecture of R. Laver.

Specifically, let us say that a tree $T$ is homogeneous if $T$ has a unique root and there exists an integer $b\meg 2$ such that every
$t\in T$ has exactly $b$ immediate successors. We show that for every $d\meg 1$ and every tuple $(T_1,...,T_d)$ of homogeneous trees, 
if $D$ is a subset of the level product of $(T_1,...,T_d)$ satisfying
\[ \limsup_{n\to\infty}
\frac{|D\cap \big( T_1(n)\times ... \times T_d(n)\big)|}{|T_1(n)\times ... \times T_d(n)|}>0\]
then there exist strong subtrees $(S_1, ..., S_d)$ of $(T_1,...,T_d)$ having common level set
such that the level product of $(S_1,...,S_d)$ is a subset of $D$.
\end{abstract}


\section{Introduction}

\subsection{Statement of the problem and the main result}

Ramsey Theory is the collection of a number of partition results asserting that for every finite coloring of a ``structure" one can find a
``substructure" which is monochromatic. In several cases, however, one can actually prove a significantly stronger \textit{density} result
asserting that every large subset of a ``structure" must contain a ``substructure". This phenomenon, investigated from the early beginnings
of Ramsey Theory, has seen some dramatic developments in recent years and, by now, there are several results in this direction. A well-known
and illuminative example is the density version of the Hales--Jewett Theorem \cite{HJ} obtained by H. Furstenberg and Y. Katznelson in \cite{FK}
(see, also, \cite{DKT2}).

The main goal of the present paper is to prove a density version of the Halpern--L\"{a}uchli Theorem \cite{HL}. The Halpern--L\"{a}uchli Theorem
is a rather deep pigeonhole principle for trees. It was discovered in 1966, three years after the discovery of the Hales--Jewett Theorem, as a
result needed for the construction of a Model of Set Theory in which the Boolean Prime Ideal Theorem is true but not the full Axiom of Choice
(see \cite{HLevy}). The original proof was based on tools from Logic; since then, other proofs have been found some of which are purely
combinatorial (see \cite[\S 3]{To2} for a detailed exposition). It has been the main tool for the development of Ramsey Theory for trees,
a rich area of Combinatorics with important applications in Functional Analysis and Topology (see, for example,
\cite{Bl,C,CS,G,Ha,Ka,La,LSV,Mi1,Mi2,PH,To2} and \cite{ADK1,ADK2,D,DK,To1} for applications).

The Halpern--L\"{a}uchli Theorem has several equivalent forms (see \cite[\S 3.1]{To2}). To proceed with our discussion it is useful
at this point to recall one of these forms, known as the ``strong subtree version of the Halpern--L\"{a}uchli Theorem".
\begin{thm} \label{t11}
For every $d\meg 1$ we have that $\hl(d)$ holds, i.e. for every tuple $(T_1,...,T_d)$ of uniquely rooted and finitely branching
trees without maximal nodes and for every finite coloring of the level product
\[  \bigcup_{n\in\nn} T_1(n)\times ...\times T_d(n) \]
of $(T_1,...,T_d)$ there exist strong subtrees $(S_1, ..., S_d)$ of $(T_1,...,T_d)$ having common level set such that the level
product of $(S_1,...,S_d)$ is monochromatic.
\end{thm}
We recall that a subtree $S$ of a tree $(T,<)$ is said to be \textit{strong} if: (a) $S$ is uniquely rooted, (b) there exists an
infinite subset $L_T(S)=\{l_0< l_1< ...\}$ of $\nn$, called the \textit{level set} of $S$, such that for every $n\in\nn$ the $n$-level
$S(n)$ of $S$ is a subset of $T(l_n)$, and (c) for every $s\in S$ and every immediate successor $t$ of $s$ in $T$ there exists a unique
immediate successor $s'$ of $s$ in $S$ with $t\mik s'$. The last condition is the most important one and expresses a basic combinatorial
requirement, namely that a strong subtree of $T$ must respect the ``tree structure" of $T$.  The notion of a strong subtree was highlighted
with the work of K. Milliken \cite{Mi1,Mi2} who used Theorem \ref{t11} to show that the family of strong subtrees of a uniquely rooted and
finitely-branching tree is partition regular.

The natural problem whether there exists a density version of Theorem \ref{t11} was first asked by R. Laver in the late 1960s who actually
conjectured that there is such a version. The conjecture was circulated among experts in the area and it was explicitly stated in the paper
\cite{BV} by R. Bicker and B. Voigt who made two important observations. Firstly, by providing several examples -- see, in particular,
\cite[Theorems 2.4 and 2.5]{BV} -- they isolated the largest class of trees for which a density version of Theorem \ref{t11} could be true.
This is the class of \textit{homogeneous} trees: a tree $T$ is said to be homogeneous if it has a unique root and there exists $b\meg 2$, 
called the \textit{branching number} of $T$, such that every $t$ in $T$ has exactly $b$ immediate successors. Secondly, they showed that
for a single homogeneous tree Theorem \ref{t11} does have a density version. Specifically, it was shown in \cite[Theorem 2.3]{BV} that
for every homogeneous tree $T$ and every subset $D$ of $T$ satisfying
\[ \limsup_{n\to\infty} \frac{|D\cap T(n)|}{|T(n)|}> 0\]
there exists a strong subtree $S$ of $T$ with $S\subseteq D$.

Our main result shows that a density version of Theorem \ref{t11} is valid for an arbitrary finite number of homogeneous trees and thereby
settles in the affirmative Laver's conjecture.
\begin{thm} \label{t12}
For every $d\meg 1$ we have that $\dhl(d)$ holds, i.e. for every tuple $(T_1,...,T_d)$ of homogeneous
trees and every subset $D$ of the level product of $(T_1,...,T_d)$ satisfying
\[ \limsup_{n\to\infty} \frac{|D\cap \big( T_1(n)\times ... \times T_d(n)\big)|}{|T_1(n)\times ... \times T_d(n)|}>0 \]
there exist strong subtrees $(S_1, ..., S_d)$ of $(T_1,...,T_d)$ having common level set such that
the level product of $(S_1,...,S_d)$ is a subset of $D$.
\end{thm}
Notice that the strong subtrees $S_1,...,S_d$ obtained by Theorem \ref{t12} are \textit{infinite}. This is the first result in Ramsey Theory
where a density condition yields the existence of an infinite object instead of a sequence of finite objects of arbitrarily large cardinality.

\subsection{Outline of the argument}

The proof of Theorem \ref{t12} proceeds by induction on the number of trees and is based on combinatorial tools. In particular, at the process
of establishing $\dhl(d+1)$ we use, as pigeonhole principles, $\dhl(d)$ as well as Theorem \ref{t11}. One can actually determine which instance
of Theorem \ref{t11} is needed in order to prove Theorem \ref{t12} for a fixed tuple $(T_1,..., T_{d+1})$ of homogeneous trees: if $b_i$ is the
branching number of $T_i$ for every $i\in\{1,...,d+1\}$, then one needs to use $\hl\big(\sum_{i=1}^{d} b_i\big)$.

Let us briefly discuss the main steps (for unexplained terminology and notation we refer to \S 2). Assume that we have proven $\dhl(d)$ for some
$d\meg 1$ and that we are given a tuple $(T_1,...,T_d,W)$ of homogeneous trees, a constant $0<\ee\mik 1$ and a subset $D$ of the level product of
$(T_1,...,T_d,W)$ satisfying
\[ |D\cap \big( T_1(n)\times ...\times T_d(n)\times W(n)\big)| \meg \ee |T_1(n)\times ... \times T_d(n)\times W(n)| \]
for infinitely many $n\in\nn$. Using a Fubini-type argument and $\dhl(d)$, we can find a vector strong subtree $\bfcs$ of $(T_1,...,T_d)$,
with the following property: there exists a strictly increasing sequence $(l_n)$ in $\nn$ such that for every $n\in\nn$ and every
$\bfs\in \otimes\bfcs(n)$ the section $D(\bfs)=\{w\in W: (\bfs,w)\in D\}$ of $D$ at $\bfs$ is a subset of $W(l_n)$ of cardinality at least
$\ee/2 |W(l_n)|$. This property of the section map $D:\otimes\bfcs\to 2^{W}$ is abstracted in Definition \ref{d31} in the main text. We call
such maps \textit{dense level selections}.

The next step (which is the most demanding part of the proof) is to show that for every dense level selection $D:\otimes\bfcs\to 2^W$ there exists
a vector strong subtree $\bfcr$ of $\bfcs$ such that the sets $\{D(\bfr): \bfr\in \otimes\bfcr\}$ are mutually ``correlated"; this is the content
of Theorem \ref{t33} in the main text. Precisely, we show that there exist a constant $0<\theta\mik 1$, a vector strong subtree $\bfcr$ of $\bfcs$
and a node $w\in D(\bfr)$, where $\bfr$ is the root of $\bfcr$, such that for every vector strong subtree $\bfcz$ of $\bfcr$ with the same root as
$\bfcr$ the density of the set
\[ \bigcap_{\bfz\in\otimes \bfcz(1)} D(\bfz) \]
relative to \textit{every} immediate successor $w'$ of $w$, is at least $\theta$. The main difficulty in the proof of this result lies in the fact
that the number of sets in the above intersection increases \textit{exponentially} with respect to the dimension\footnote[1]{A typical phenomenon
in the proof of several combinatorial results is that the ``low dimensional" cases are relatively easy to prove and the full complexity appears
after a critical threshold. In the case of the density Halpern--L\"{a}uchli Theorem this critical threshold is dimension $3$. In particular, the
authors are aware of a different, and in a sense more effective, proof of $\dhl(2)$. This proof, however, cannot be generalized to higher
dimensions.}. It is worth pointing out that in this step we use again $\dhl(d)$ as pigeonhole principle, but in a slightly different form
(Proposition \ref{p35} in the main text).

With Theorem \ref{t33} at hand, one can perform a recursive construction in order to find a vector strong subtree $(Z_1,....,Z_d,V)$ of
$(T_1,...,T_d,W)$ whose level product is a subset of $D$. This recursive selection, however, is rather unusual since we actually construct
an infinite chain of $(T_1,...,T_d)$ and a strong subtree $V$ of $W$ with special properties. The desired vector strong subtree is then obtained
using an ``unfolding" argument.

\subsection{Organization of the paper}

The paper is organized as follows. In \S 2 we set up our notation and terminology. The next section is devoted to the study of a natural class of
finite vector trees, which we call \textit{vector fans}. In \S 4 we introduce the notion of a dense level selection we mentioned above; the main
result in this section is Theorem \ref{t33}. The proof of Theorem \ref{t12} is completed in \S 5. Finally, in \S 6 we make some comments. 

\subsection{Acknowledgments}

We would like to thank S. Todorcevic for his comments and remarks as well as for several historical information concerning the problem solved
in the paper. The first named author was supported by NSF grant DMS-0903558.


\section{Background material}

By $\nn=\{0,1,2,...\}$ we denote the natural numbers. The cardinality of a set $X$ will be denoted by $|X|$.

\subsection{Trees and subtrees}

By the term \textit{tree} we mean a partially ordered set $(T,<)$ such that the set $\{s\in T: s<t\}$ is finite and
linearly ordered under $<$ for every $t\in T$. The cardinality of this set is defined to be the \textit{length of $t$
in $T$} and is denoted by $\ell_T(t)$. For every $n\in\nn$ the \textit{$n$-level of $T$}, denoted by $T(n)$, is
defined to be the set $\{t\in T: \ell_T(t)=n\}$. The \textit{height} of $T$, denoted by $h(T)$, is defined as follows.
If there exists $k\in\nn$ with $T(k)=\varnothing$, then we set $h(T)=\max\{n\in\nn: T(n)\neq\varnothing\}+1$;
otherwise, we set $h(T)=\infty$.

For every $t\in T$ by $\suc_T(t)$ we denote the set of \textit{successors of $t$ in $T$}, i.e.
\begin{equation} \label{e2new}
\suc_T(t)=\{s\in T: t\mik s\}.
\end{equation}
The \textit{set of immediate successors of $t$ in $T$} is the subset of $\suc_T(t)$ defined by
$\immsuc_T(t)=\{s\in T: t\mik s \text{ and } \ell_T(s)=\ell_T(t)+1\}$. More generally, for every
subset $F$ of $T$ we set $\suc_T(F)=\{s\in T: \text{exists } t\in F \text{ with } t\mik s\}$.

Let $n\in\nn$ with $n< h(T)$ and $F\subseteq T(n)$. The \textit{density} of $F$ is defined by
\begin{equation} \label{e21}
\dens(F)=\frac{|F|}{|T(n)|}.
\end{equation}
More generally, for every $m\in\nn$ with $m\mik n$ and every $t\in T(m)$ the \textit{density of $F$ relative to the node $t$}
is defined by
\begin{equation} \label{e22}
\dens(F \ | \ t)=\frac{|F\cap \suc_T(t)|}{|T(n)\cap \suc_T(t)|}.
\end{equation}

A \textit{subtree} $S$ of a tree $(T,<)$ is a subset of $T$ viewed as a tree equipped with the induced partial
ordering. For every $n\in\nn$ with $n<h(T)$ we set
\begin{equation} \label{e23}
T\upharpoonright n= T(0)\cup ... \cup T(n).
\end{equation}
Notice that $h(T\upharpoonright n)=n+1$. An \textit{initial subtree} of $T$ is a subtree of $T$ of the form
$T\upharpoonright n$ for some $n\in\nn$.

Finally, we recall that a tree $T$ is said to be \textit{pruned} (respectively, \textit{finitely branching}) if for
every $t\in T$ the set of immediate successors of $t$ in $T$ is nonempty (respectively, finite). It is said to be
\textit{uniquely rooted} if $|T(0)|=1$. The \textit{root} of a uniquely rooted tree $T$ is defined to be the node $T(0)$.

\subsection{Vector trees and level products}

A \textit{vector tree} $\bfct$ is a nonempty finite sequence of trees having common height; this common height
is defined to be the \textit{height} of $\bfct$ and will be denoted by $h(\bfct)$. We notice that, throughout the paper,
we will start the enumeration of vector trees with $1$ instead of $0$.

For every vector tree $\bfct=(T_1, ...,T_d)$ and every $n\in\nn$ with $n< h(\bfct)$ we set
\begin{equation} \label{e24}
\bfct\upharpoonright n=(T_1\upharpoonright n, ..., T_d\upharpoonright n).
\end{equation}
A vector tree of this form is called a \textit{vector initial subtree} of $\bfct$. Also let
\begin{equation} \label{e25}
\bfct(n)=\big(T_1(n),...,T_d(n)\big)
\end{equation}
and
\begin{equation} \label{e26}
\otimes\bfct(n)= T_1(n)\times ...\times T_d(n).
\end{equation}
The \textit{level product of $\bfct$}, denoted by $\otimes\bfct$, is defined to be the set
\begin{equation} \label{e27}
\bigcup_{n< h(\bfct)} \otimes\bfct(n).
\end{equation}
If $\bft=(t_1,...,t_d)\in\otimes\bfct$, then we define $\ell_{\bfct}(\bft)$ to be the unique $n\in\nn$ such that $\bft\in\otimes\bfct(n)$.
Also we set
\begin{equation} \label{e2newstaff}
\suc_{\bfct}(\bft)=\big( \suc_{T_1}(t_1), ..., \suc_{T_d}(t_d)\big).
\end{equation}
Finally, we say that a vector tree $\bfct=(T_1,...,T_d)$ is \textit{pruned} (respectively, \textit{finitely branching},
\textit{uniquely rooted}) if $T_i$ is pruned (respectively, finitely branching, uniquely rooted) for every $i\in \{1,...,d\}$.
Notice that if $\bfct$ is uniquely rooted, then $\bfct(0)=\otimes\bfct(0)$; the element $\bfct(0)$ will be called the
\textit{root} of $\bfct$.

\subsection{Strong subtrees and vector strong subtrees}

Let $T$ be a pruned, finitely branching and uniquely rooted tree. A subtree $S$ of $T$ is said to be \textit{strong} provided that:
(a) $S$ is uniquely rooted, (b) for every $n\in\nn$ there exists $m\in\nn$ with $S(n)\subseteq T(m)$, and (c) for every
$s\in S$ and every $t\in\immsuc_T(s)$ there exists a unique node $s'\in\immsuc_S(s)$ such that $t\mik s'$.
The \textit{level set} of a strong subtree $S$ of $T$ is defined to be the set
\begin{equation} \label{e28}
L_T(S)=\{ m\in\nn: \text{exists } n\in\nn \text{ with } S(n)\subseteq T(m)\}.
\end{equation}
A \textit{finite strong subtree} of $T$ is an initial subtree of a strong subtree of $T$.

The above concepts are naturally extended to vector trees. Specifically, let $\bfct=(T_1,...,T_d)$ be a pruned, finitely branching
and uniquely rooted vector tree. A \textit{vector strong subtree of} $\bfct$ is a vector tree $\bfcs=(S_1,...,S_d)$ such that
$S_i$ is a strong subtree of $T_i$ for every $i\in\{1,...,d\}$ and $L_{T_1}(S_1)= ...= L_{T_d}(S_d)$. A \textit{finite vector
strong subtree of} $\bfct$ is a vector initial subtree of a vector strong subtree of $\bfct$.

\subsection{Homogeneous trees and vector homogeneous trees}

Let $b\in\nn$ with $b\meg 2$. By $b^{<\nn}$ we shall denote the set of all finite sequences having values in $\{0,...,b-1\}$.
The empty sequence is denoted by $\varnothing$ and is included in $b^{<\nn}$. We view $b^{<\nn}$ as a tree equipped with the
(strict) partial order $\sqsubset$ of end-extension. For every $n\in\nn$ by $b^n$ we denote the $n$-level of $b^{<\nn}$.
If $n\meg 1$, then $b^{<n}$ stands for the initial subtree of $b^{<\nn}$ of height $n$. By $\lex$ we denote the usual
lexicographical order on $b^n$. For every $t,s\in b^{<\nn}$ by $t^{\con}s$ we denote their concatenation.

As we have already mentioned in the introduction, a homogeneous tree $T$ is a uniquely rooted tree such that every node in $T$ has
exactly $b$ immediate successors, where $b\meg 2$ is the branching number of $T$. In several cases, we need to enumerate the set
of nodes of a level of $T$. There is, of course, no problem for selecting an enumeration. But an arbitrary enumeration might
lack compatibility when passing to subtrees. This problem can be resolved by restricting our attention to the class of strong
subtrees of a \textit{fixed} homogeneous tree. It is, of course, clear that all homogeneous trees with the same branching
number are pairwise isomorphic, and so, such a restriction will have no effect in the generality of our results.
\medskip

\noindent \textbf{Convention.} \textit{In the rest of the paper by the term ``homogeneous tree" we will mean a strong subtree
of $b^{<\nn}$ for some $b\in\nn$ with $b\meg 2$. For every homogeneous tree $T$ by $b_T$ we shall denote the branching number
of $T$ and we set $B_T=b_T^{<\nn}$. We follow the same conventions for vector trees. Precisely, by the term ``vector homogeneous
tree" we will mean a vector strong subtree of $(b_1^{<\nn},...,b_d^{<\nn})$ for some $b_1,...,b_d\in\nn$ with $b_i\meg 2$ for
every $i\in\{1,...,d\}$. For every vector homogeneous tree $\bfct=(T_1,...,T_d)$ we set $b_{\bfct}=(b_{T_1},...,b_{T_d})$ and
$\bfcb_{\bfct}=(b_{T_1}^{<\nn},...,b_{T_d}^{<\nn})$.}
\medskip

The above convention enables us to effectively enumerate the set of immediate successors of a given node of a homogeneous tree $T$. 
Specifically, for every $t\in T$ and every $p\in\{0,...,b_T-1\}$ we set
\begin{equation} \label{e241}
t^{\con_T}\!p= \immsuc_T(t)\cap \suc_{B_T}(t^{\con}p).
\end{equation}
Notice that
\[ \immsuc_T(t)=\big\{ t^{\con_T}\!p: p\in\{0,...,b_T-1\}\big\}. \]
Also observe that for every $p,q\in\{0,...,b_T-1\}$ we have $t^{\con_T}\!p \lex t^{\con_T}\!q$ if and only if $p<q$.

\subsection{Canonical embeddings and vector canonical embeddings}

Let $T$ and $S$ be two homogeneous trees with the same branching number. We will say that a map $f:T\to S$ is a \textit{canonical embedding}
if the following conditions are satisfied.
\begin{enumerate}
\item[(a)] For every $t,t'\in T$ we have $\ell_T(t)=\ell_T(t')$ if and only if $\ell_S\big(f(t)\big)=\ell_S\big(f(t')\big)$.
\item[(b)] For every $t\in T$ and $p\in\{0,...,b_T-1\}$ we have $f(t^{\con_T}\!p)\in\suc_S\big(f(t)^{\con_S}\!p\big)$.
\end{enumerate}
It is easy to verify that for every canonical embedding $f:T\to S$ the following hold: (i) for every $t,t'\in T$ we have $t\sqsubset t'$
if and only if $f(t)\sqsubset f(t')$, (ii) $f$ is an injection, and (iii) the image $f(T)$ of $T$ under $f$ is a strong subtree of $S$.

Also notice that there exists a \textit{unique} bijection between $T$ and $S$ satisfying the above conditions. This unique bijection will be
called the \textit{canonical isomorphism} between $T$ and $S$ and will be denoted by $\ci(T,S)$.

We proceed to define the notion of a ``vector canonical embedding". It is a kind of ``tensorization" of a finite sequence of canonical
embeddings with special properties. Specifically, let $\bfct=(T_1,...,T_d)$ and $\bfcs=(S_1,...,S_d)$ be two vector homogeneous trees with
$b_\bfct=b_\bfcs$. For every $i\in\{1,...,d\}$ let $f_i:T_i\to S_i$ be a canonical embedding and assume that for every $n\in\nn$ and
every $\bft=(t_1,...,t_d)\in\otimes\bfct(n)$ we have $\ell_{S_1}\big(f_1(t_1)\big)= ... = \ell_{S_d}\big(f_d(t_d)\big)$.
This assumption permits us to define a map $(\otimes_{i=1}^d f_i):\otimes\bfct\to\otimes\bfcs$ by the rule
\begin{equation} \label{e251}
(\otimes_{i=1}^d f_i)\big((t_1,...,t_d)\big)=\big(f_1(t_1), ..., f_d(t_d)\big).
\end{equation}
A map of this form will be called a \textit{vector canonical embedding} of $\otimes\bfct$ into $\otimes\bfcs$.
The \textit{vector canonical isomorphism} between $\otimes\bfct$ and $\otimes\bfcs$ is defined to be the map
$(\otimes_{i=1}^d \ci(T_i,S_i))$ and will be denoted by $\bfci(\bfct,\bfcs)$.


\section{Fans and vector fans}

We start with the following.
\begin{defn} \label{dns1}
Let $T$ be a homogeneous tree. We say that a tree $F$ is a \emph{fan of $T$} if $F$ is of the form $R\upharpoonright 1$
for some strong subtree $R$ of $T$. The set of all fans of $T$ will be denoted by $\fan(T)$.
\end{defn}
Next we introduce the higher-dimensional analogues of fans.
\begin{defn} \label{dns2}
Let $\bfct$ be a vector homogeneous tree. We say that vector tree $\bfcf$ is a \emph{vector fan of $\bfct$} if $\bfcf$
is of the form $\bfcr\upharpoonright 1$ for some vector strong subtree $\bfcr$ of $\bfct$. The set of all vector fans of
$\bfct$ will be denoted by $\fan(\bfct)$.
\end{defn}
We view vector fans as the fundamental building blocks of vector homogeneous trees. This point of view is crucial for the
proof of Theorem \ref{t12}. Also we make two simple observations. Firstly, we notice that if $\bfcr$ is a vector strong subtree
of a vector homogeneous tree $\bfct$, then $\fan(\bfcr)\subseteq\fan(\bfct)$. Secondly, we observe that if $\bfcf$ is a vector
fan of $\bfct$, then $\bfcf(0)\in\otimes\bfct$ and $\otimes\bfcf(1)\subseteq \otimes\bfct$.

We will need two combinatorial results concerning vector fans. The first one is the following (see \cite[Theorem 1.3]{Mi1}).
\begin{prop} \label{pns3}
Let $\bfct$ be a vector homogeneous tree and set $\bft_0=\bfct(0)$. Then for every finite coloring
\[ \fan(\bfct)=\mathcal{C}_0\cup ... \cup \mathcal{C}_r \]
there exist $m\in\{0,...,r\}$ and a vector strong subtree $\bfcz$ of $\bfct$ with $\bfcz(0)=\bft_0$
such that $\bfcf\in\mathcal{C}_m$ for every $\bfcf\in\fan(\bfcz)$ with $\bfcf(0)=\bft_0$.
\end{prop}
To state the second result we need to introduce some notation. For every vector
homogeneous tree $\bfct$ and every $n\in\nn$ with $n\meg 1$ we set
\begin{equation} \label{ens1}
\fan(\bfct,n)=\big\{ \bfcf\in \fan(\bfct): \otimes\bfcf(1)\subseteq\otimes\bfct(n)\big\}.
\end{equation}
\begin{prop} \label{pns4}
Let $\bfct=(T_1,...,T_d)$ be a vector homogeneous tree. For every $n\in\nn$ with $n\meg 1$ let $\mathcal{F}_n$ be a subset
of $\fan(\bfct)$ with the following property.
\begin{enumerate}
\item[(P)] For every vector strong subtree $\bfcs$ of $\bfct$ there exists $\bfcf\in\fan(\bfcs)\cap \mathcal{F}_n$
with $\bfcf(0)=\bfcs(0)$.
\end{enumerate}
Also let $\bfcr$ be a vector strong subtree $\bfct$ and set $\bfr_0=\bfcr(0)$. Then there exists a vector strong subtree $\bfcz$
of $\bfcr$ with $\bfcz(0)=\bfr_0$ such that for every $n\in\nn$ with $n\meg 1$ and every $\bfcf\in\fan(\bfcz,n)$ with
$\bfcf(0)=\bfr_0$ we have $\bfcf\in\mathcal{F}_n$.
\end{prop}
Proposition \ref{pns4} can be hardly characterized as new since it follows using fairly standard arguments (see, e.g.,
\cite{C,Mi1,Mi2}). Nevertheless, we have decided to include a proof for two reasons. The first one is for self-containedness. 
Secondly, because we want to emphasize which instance of Theorem \ref{t11} is needed for the proof.
\begin{proof}[Proof of Proposition \ref{pns4}]
We write $\bfcr=(R_1,...,R_d)$ and $\bfr_0=(r_1,...,r_d)$. Recursively, we shall construct a sequence $(\bfcr_n)$
of vector strong subtrees of $\bfcr$ such that for every $n\in\nn$ the following are satisfied.
\begin{enumerate}
\item[(a)] $\bfcr_n(0)=\bfr_0$.
\item[(b)] $\bfcr_{n+1}$ is a vector strong subtree of $\bfcr_n$ and $\bfcr_{n+1}\upharpoonright n=\bfcr_n\upharpoonright n$.
\item[(c)] If $n\meg 1$, then for every $\bfcf\in\fan(\bfcr_n,n)$ with $\bfcf(0)=\bfr_0$ we have $\bfcf\in\mathcal{F}_n$.
\end{enumerate}
Assuming that the construction has been carried out, we define $\bfcz$ to be the unique vector strong subtree of $\bfcr$
satisfying $\bfcz(n)=\bfcr_n(n)$ for every $n\in\nn$. It is easily seen that $\bfcz$ is the desired vector tree.

We proceed to the construction. For $n=0$ we set $\bfcr_0=\bfcr$ and we notice that with this choice property (a)
is satisfied (the other properties are meaningless for $n=0$). Assume that for some $n\in\nn$ we have constructed the vector 
trees $\bfcr_0,...,\bfcr_n$ so that (a), (b) and (c) are satisfied and write $\bfcr_n=(R^n_1,...,R^n_d)$. For the construction
the vector tree $\bfcr_{n+1}$ we need to introduce some notation and terminology.
\medskip

\noindent \textbf{(A)} Let $i\in\{1,...,d\}$ be arbitrary. We set $M_i=b_{T_i}^{n+1}$ and we notice that the cardinality of the
$(n+1)$-level $R^n_i(n+1)$ of $R^n_i$ is $M_i$. We write the set $R^n_i(n+1)$ in lexicographical increasing order as
$\{t^i_1 \lex ...\lex t^i_{M_i}\}$ and we set $V^i_j=\suc_{R^n_i}(t^i_j)$ for every $j\in\{1,...,M_i\}$.
\medskip

\noindent \textbf{(B)} We define $\bfcv$ to be the vector tree $(V^i_j)_{i=1, j=1}^{d \ \ \ M_i}$. For every vector strong subtree
$\bfcu$ of $\bfcv$ we can naturally associate a vector strong subtree $\bfcr^{\bfcu}=(R^{\bfcu}_1,...,R^{\bfcu}_d)$ of $\bfcr_n$.
Precisely, write $\bfcu$ as $(U^i_j)_{i=1, j=1}^{d \ \ \ M_i}$ and for every $i\in\{1,...,d\}$ set
\[ R^{\bfcu}_i=(R^n_i\upharpoonright n)\cup U^i_1\cup ... \cup U^i_{M_i}.\]
Observe that $\bfcr^{\bfcu}\upharpoonright n=\bfcr_n\upharpoonright n$. Also notice that $\bfcr^\bfcv=\bfcr_n$.
\medskip

\noindent \textbf{(C)} Next we introduce the notion of a \textit{strong position}. It is a technical tool for the construction
of the vector tree $\bfcr_{n+1}$. Specifically, a strong position $\mathcal{P}$ is defined to be a finite sequence $(P_1,...,P_d)$ such that
\begin{enumerate}
\item[(I)] $P_i\subseteq \{1,...,M_i\}$ for every $i\in\{1,...,d\}$, and
\item[(II)] if $F_i=\{r_i\}\cup\{ t^i_j:j\in P_i\}$ for every $i=\{1,...,d\}$, then $\bfcf=(F_1,...,F_d)$ is a vector fan of $\bfcr_n$.
\end{enumerate}
By (II), if $\mathcal{P}=(P_1,...,P_d)$ is a strong position, then $|P_i|=b_{T_i}$ for every $i\in\{1,...,d\}$.
For every $\bfv=(v^i_j)_{i=1, j=1}^{d \ \ \ M_i}\in\otimes\bfcv$ and every strong position $\mathcal{P}=(P_1,...,P_d)$ we set
\[ \bfcf_{\bfv,\mathcal{P}}=\big( \{r_1\}\cup \{v^1_j:j\in P_1\}, ..., \{r_d\}\cup \{v^d_j:j\in P_d\}\big).\]
By (II), it is clear that $\bfcf_{\bfv,\mathcal{P}}\in\fan(\bfcr^{\bfcv})$ and $\bfcf_{\bfv,\mathcal{P}}(0)=\bfr_0$.
We isolate the following fact: \textit{if $\bfcu$ is a vector strong subtree of $\bfcv$, $k\in\nn$ with $k\meg n+1$
and $\bfcf\in\fan(\bfcr^{\bfcu},k)$ with $\bfcf(0)=\bfr_0$, then there exist a unique strong position $\mathcal{P}$
and an element $\bfu$ of $\otimes\bfcu$ (not necessarily unique) such that $\bfcf=\bfcf_{\bfu,\mathcal{P}}$}. The existence
of $\mathcal{P}$ and $\bfu$ is a rather direct consequence of the relevant definitions.
\medskip

After this preliminary discussion we are ready to proceed to the construction of the vector tree $\bfcr_{n+1}$. For every strong
position $\mathcal{P}$ let
\[ \mathcal{G}_{\mathcal{P}}=\big\{ \bfv\in\otimes\bfcv: \bfcf_{\bfv,\mathcal{P}}\in\mathcal{F}_{n+1}\big\}. \]
Applying successively Theorem \ref{t11}, we find a vector strong subtree $\bfcu_0$ of $\bfcv$ such that for every strong
position $\mathcal{P}$ we have that either $\otimes\bfcu_0\subseteq\mathcal{G}_{\mathcal{P}}$ or 
$\otimes\bfcu_0\cap \mathcal{G}_{\mathcal{P}}=\varnothing$. Notice that the set $\mathcal{G}_{\mathcal{P}}$ depends only on the
coordinates determined by $\mathcal{P}$, and so, each time we need to apply $\hl\big(\sum_{i=1}^d b_{T_i}\big)$.

We set $\bfcr_{n+1}=\bfcr^{\bfcu_0}$. The vector tree $\bfcr_{n+1}$ is the desired one. It is clear that we only
need to check that property (c) is satisfied. So, let $\bfcf\in\fan(\bfcr_{n+1},n+1)$ with $\bfcf(0)=\bfr_0$ be arbitrary.
As we have already mentioned in \textbf{(C)} above, there exist a unique strong position $\mathcal{Q}=(Q_1,...,Q_d)$
and an element $\bfu$ of $\otimes\bfcu_0$ (not necessarily unique) such that $\bfcf=\bfcf_{\bfu,\mathcal{Q}}$.
In order to show that $\bfcf\in\mathcal{F}_{n+1}$ it is enough to prove that $\otimes\bfcu_0\subseteq\mathcal{G}_{\mathcal{Q}}$.
To this end, we will argue by contradiction. So, assume that $\otimes\bfcu_0\cap\mathcal{G}_{\mathcal{Q}}=\varnothing$.
We write $\bfcu_0$ as $(U^i_j)_{i=1, j=1}^{d \ \ \ M_i}$ and for every $i\in\{1,...,d\}$ we set $S_i=\{r_i\} \cup \{U^i_j:j\in Q_i\}$.
Let $\bfcs=(S_1,...,S_d)$ and notice that $\bfcs$ is a vector strong subtree of $\bfct$ with $\bfcs(0)=\bfr_0$.
Let $\bfcf'\in\fan(\bfcs)$ with $\bfcf'(0)=\bfr_0$. Observe that there exists $k\in\nn$ with $k\meg n+1$ such that
$\bfcf'\in\fan(\bfcr^{\bfcu_0},k)$. Hence, there exists an element $\bfu'\in\otimes\bfcu_0$ (not necessarily unique)
such that $\bfcf'=\bfcf_{\bfu',\mathcal{Q}}$. Since $\bfu'\in\otimes\bfcu_0$ we see that $\bfu'\notin\mathcal{G}_{\mathcal{Q}}$
and so $\bfcf'\notin\mathcal{F}_{n+1}$. In other words, for every $\bfcf'\in\fan(\bfcs)$ with $\bfcf'(0)=\bfcs(0)$ we have that
$\bfcf'\notin\mathcal{F}_{n+1}$. This contradicts property (P). Therefore, $\otimes\bfcu_0\subseteq\mathcal{G}_{\mathcal{Q}}$
and so the vector tree $\bfcr_{n+1}$ has the desired properties.

This completes the recursive construction, and as we have already indicated, the proof of Proposition \ref{pns4} is
also completed.
\end{proof}


\section{Dense level selections}

\subsection{Definitions and statement of the main result}

We start by introducing the following definition.
\begin{defn} \label{d31}
Let $\bfct=(T_1,...,T_d)$ be a vector homogeneous tree, $W$ a homogeneous tree and $0<\ee\mik 1$. We say that a map
$D:\otimes\bfct\to 2^W$ is an \emph{$\ee$-dense level selection} if there exists a strictly increasing sequence $(l_n)$
in $\nn$ such that for every $n\in\nn$ and every $\bft\in\otimes\bfct(n)$ we have $D(\bft)\subseteq W(l_n)$ and
$\dens\big(D(\bft)\big)\meg\ee$.
\end{defn}
The next definition is a crucial conceptual step towards the proof of Theorem \ref{t12}.
\begin{defn} \label{d32}
Let $\bfct=(T_1,...,T_d)$ be a vector homogeneous tree, $W$ a homogeneous tree, $0<\ee\mik 1$ and $D:\otimes\bfct\to 2^W$
an $\ee$-dense level selection. Also let $\bfcr$ be a vector strong subtree of $\bfct$, $w\in W$ and $0<\theta\mik 1$.
We say that the pair $(\bfcr,w)$ is \emph{strongly $\theta$-correlated with respect to} $D$ if, setting $\bfr_0=\bfcr(0)$,
the following conditions are satisfied.
\begin{enumerate}
\item[(C1)] We have $w\in D(\bfr_0)$.
\item[(C2)] For every $\bfcf\in\fan(\bfcr)$ with $\bfcf(0)=\bfr_0$ and every $p\in\{0,...,b_W-1\}$ we have
\begin{equation} \label{e311}
\dens\Big( \bigcap_{\bfr\in\otimes \bfcf(1)} D(\bfr) \ \big| \ w^{\con_W}\!p \, \Big)\meg \theta.
\end{equation}
\end{enumerate}
\end{defn}
We are now ready to state the main result in this section.
\begin{thm} \label{t33}
Let $d\meg 1$ and assume that $\mathrm{DHL}(d)$ holds. Also let $\bfct=(T_1,...,T_d)$ be a vector homogeneous tree,
$W$ a homogeneous tree, $0<\ee\mik 1$ and $D:\otimes\bfct\to 2^W$ an $\ee$-dense level selection. Then there exist
a vector strong subtree $\bfcr$ of $\bfct$, $w\in W$ and $0<\theta\mik 1$ such that the pair $(\bfcr,w)$ is strongly
$\theta$-correlated with respect to $D$.
\end{thm}
The proof of Theorem \ref{t33} will be given in \S 4.3. At this point, let us isolate the following consequence
of Theorem \ref{t33}. It will be of particular importance in \S 5.
\begin{cor} \label{c34}
Let $d\meg 1$ and assume that $\mathrm{DHL}(d)$ holds. Also let $\bfct=(T_1,...,T_d)$ be a vector homogeneous tree, $W$
a homogeneous tree, $0<\ee\mik 1$ and $D:\otimes\bfct\to 2^W$ an $\ee$-dense level selection. Then there exist a vector strong
subtree $\bfcs$ of $\bfct$ and for every $\bfs\in\otimes\bfcs$ a node $w_{\bfs}\in W$ and a constant $0<\theta_{\bfs}\mik 1$
with the following property. For every $\bfs\in\otimes\bfcs$ and every vector strong subtree $\bfcz$ of $\suc_{\bfcs}(\bfs)$
with $\bfcz(0)=\bfs$ the pair $(\bfcz,w_\bfs)$ is strongly $\theta_{\bfs}$-correlated with respect to $D$.
\end{cor}
\begin{proof}
We start with the following observation. Let $\bfcr$ be a vector strong subtree of $\bfct$, $w\in W$ and $0<\theta\mik 1$
and assume that the pair $(\bfcr,w)$ is strongly $\theta$-correlated with respect to $D$. Then for every vector strong subtree
$\bfcz$ of $\bfcr$ with $\bfcz(0)=\bfcr(0)$ the pair $(\bfcz,w)$ is also strongly $\theta$-correlated with respect to $D$.

Therefore, what we need to find is a vector strong subtree $\bfcs$ of $\bfct$, a family $\{w_\bfs:\bfs\in\otimes\bfcs\}$
in $W$ and a family $\{\theta_\bfs:\bfs\in\otimes\bfcs\}$ of reals in $(0,1]$ such that for every $\bfs\in\otimes\bfcs$ the
pair $\big(\suc_{\bfcs}(\bfs),w_\bfs\big)$ is strongly $\theta_\bfs$-correlated with respect to $D$. This can be proved using
$\hl\big(\sum_{i=1}^d b_{T_i}\big)$ as pigeonhole principle, Theorem \ref{t33} and arguing as in the proof of Proposition
\ref{pns4}. We prefer, however, to give a very simple proof which is based on Theorem \ref{t33} and on the work of K. Milliken
on Ramsey properties of strong subtrees. For every vector strong subtree $\bfcz$ of $\bfct$ let $[\bfcz]_{\strong}$
be the set of all vector strong subtrees of $\bfcz$ and notice that $[\bfcz]_{\strong}$ is $G_{\delta}$ (hence Polish) subspace
of $2^{T_1}\times ... \times 2^{T_d}$. Now let $\mathcal{C}$ be the subset of $[\bfct]_{\strong}$ defined by
\begin{eqnarray*}
\bfcr\in\mathcal{C} & \Leftrightarrow & \text{there exist } w\in W \text{ and } 0<\theta\mik 1 \text{ such that the pair}\\\
& & (\bfcr,w) \text{ is strongly $\theta$-correlated with respect to } D.
\end{eqnarray*}
Notice that $\mathcal{C}$ is an $F_{\sigma}$ subset of $[\bfct]_{\strong}$. Moreover, by Theorem \ref{t33}, we see that $\mathcal{C}
\cap [\bfcz]_{\strong}\neq\varnothing$ for every vector strong subtree $\bfcz$ of $\bfct$. By \cite[Theorem 2.1]{Mi2},
there exists a vector strong subtree $\bfcs$ of $\bfct$ such that $[\bfcs]_{\strong}\subseteq\mathcal{C}$.
Observing that $\suc_\bfcs(\bfs)\in [\bfcs]_{\strong}$ for every $\bfs\in\otimes\bfcs$, the result follows.
\end{proof}

\subsection{A consequence of $\dhl(d)$}

In this subsection we shall obtain a consequence of $\dhl(d)$ which is stated within the context of dense level
selections. It will be used in the proof of Theorem \ref{t33}.
\begin{prop} \label{p35}
Let $d\meg 1$ and assume that $\mathrm{DHL}(d)$ holds. Let $\bfcz=(Z_1,...,Z_d)$ be a vector homogeneous tree,
$W$ a homogeneous tree, $0<\eta\mik 1$ and $B:\otimes\bfcz\to 2^W$ an $\eta$-dense level selection. Then for every
$p\in\{0,...,b_W-1\}$ there exist $\bfcf\in\fan(\bfcz)$ and $w\in B(\bfz_0)$, where $\bfz_0=\bfcf(0)$, such that
\[ \bigcap_{\bfz\in\otimes\bfcf(1)} B(\bfz)\cap \suc_W(w^{\con_W}\!p)\neq\varnothing. \]
\end{prop}
\begin{proof}
We fix $p\in\{0,...,b_W-1\}$. Let $(l_n)$ be the strictly increasing sequence in $\nn$ such that for every $n\in\nn$
and every $\bfz\in\otimes\bfcz(n)$ we have $B(\bfz)\subseteq W(l_n)$ and $\dens\big(B(\bfz)\big)\meg \eta$.
For every $n\in\nn$ we define $C_n\subseteq W(l_n)$ by the rule
\[ w\in C_n\Leftrightarrow |\{\bfz\in\otimes\bfcz(n): w\in B(\bfz)\}|\meg \eta/2 \ |\otimes\bfcz(n)|. \]
\begin{claim} \label{fubini}
For every $n\in\nn$ we have $\dens(C_n)\meg \eta/2$.
\end{claim}
\begin{proof}[Proof of Claim \ref{fubini}]
This is a rather standard estimate and follows using a Fubini-type argument. Indeed, let
\[ E_n=\big\{ (\bfz,w)\in\otimes\bfcz(n)\times W(l_n): w\in B(\bfz)\big\}. \]
Since $\dens\big(B(\bfz)\big)\meg\eta$ for every $\bfz\in\otimes\bfcz(n)$, we have
\[ \eta \cdot |\otimes\bfcz(n)| \cdot |W(l_n)| \mik |E_n|. \]
On the other hand, by the definition of the set $C_n$, we get
\[ |E_n| \mik |C_n| \cdot |\otimes\bfcz(n)| + (\eta/2) \cdot |\otimes\bfcz(n)| \cdot |W(l_n)|. \]
Therefore, $\dens(C_n)\meg \eta/2$. The proof of Claim \ref{fubini} is completed.
\end{proof}
By Claim \ref{fubini} and \cite[Theorem 2.3]{BV}, we may find a strictly increasing sequence $(n_k)$ in $\nn$
and a sequence $(w_k)$ in $W$ such that for every $k,m\in\nn$ with $k<m$ we have
\begin{enumerate}
\item[(a)] $w_k\in C_{n_k}$ and
\item[(b)] $w_m\in \suc_W(w_k^{\con_W}\!p)$.
\end{enumerate}
We define $B'\subseteq \otimes\bfcz$ by
\[ B'=\bigcup_{k\in\nn} \big\{\bfz\in\otimes\bfcz(n_k): w_k\in B(\bfz)\big\} \]
By (a) and the definition of the set $C_{n_k}$, we see that
\[ \limsup_{n\to\infty} \frac{|B'\cap\otimes\bfcz(n)|}{|\otimes\bfcz(n)|} =
\limsup_{k\to\infty} \frac{|B'\cap\otimes\bfcz(n_k)|}{|\otimes\bfcz(n_k)|} \meg \eta/2>0. \]
Therefore, using our hypothesis that $\dhl(d)$ holds, it is possible to find a vector strong subtree $\bfcr$ of $\bfcz$
such that $\otimes\bfcr\subseteq B'$. We set $\bfcf=\bfcr\upharpoonright 1\in\fan(\bfcz)$ and $\bfz_0=\bfcf(0)$.
Let $k_0$ and $k_1$ be the unique integers such that $\bfz_0\in\otimes\bfcz(n_{k_0})$
and $\otimes\bfcf(1)\subseteq\otimes\bfcz(n_{k_1})$. Clearly $k_0<k_1$. Notice that
\[ w_{k_0}\in B(\bfz_0)\]
since $\bfz_0\in\otimes\bfcr\subseteq B'$. Moreover,
\[ w_{k_1}\in \bigcap_{\bfz\in\otimes\bfcf(1)} B(\bfz) \]
since $\otimes\bfcf(1)\subseteq\otimes\bfcr\subseteq B'$. Using (b), we conclude that
\[ w_{k_1}\in\bigcap_{\bfz\in\otimes\bfcf(1)} B(\bfz)\cap \suc_W(w_{k_0}^{\con_W}\!p). \]
The proof of Proposition \ref{p35} is completed.
\end{proof}

\subsection{Proof of Theorem \ref{t33}}

The proof is a quest of a contradiction. So, assume that there exist a vector homogeneous tree
$\bfct=(T_1,...,T_d)$, a homogeneous tree $W$, a constant $0<\ee\mik 1$ and an $\ee$-dense level selection
$D:\otimes\bfct\to 2^W$ such that
\begin{enumerate}
\item[(H)] for every vector strong subtree $\bfcr$ of $\bfct$, every $w\in W$ and every $0<\theta\mik 1$
the pair $(\bfcr,w)$ is not strongly $\theta$-correlated with respect to $D$.
\end{enumerate}
\textit{The vector homogeneous tree $\bfct$ and the $\ee$-dense level selection $D:\otimes\bfct\to 2^W$ will be fixed throughout
the proof}.  We will use hypothesis (H) to derive a contradiction. Our strategy is to construct a vector strong subtree $\bfcz$
of $\bfct$ and an $(\ee/2b_W)$-dense level selection $B:\otimes\bfcz\to 2^W$ that violates the conclusion of Proposition \ref{p35}
for some $p_0\in\{0,...,b_W-1\}$. The construction will be done in several intermediate steps. For notational simplicity, for every
$i\in\{1,...,d\}$ by $b_i$ we shall denote the branching number of the tree $T_i$.

\subsection*{Step 1: selection of a rapidly decreasing sequence}

For every $n\in\nn$ with $n\meg 1$ we define
\begin{equation} \label{e31new}
\Theta_n =|\fan(\bfct,n)|.
\end{equation}
Setting $\beta=\prod_{i=1}^d b_i$ we see that $\Theta_n \mik \Theta_{n+1}$ and $\beta^{n-1}\mik \Theta_n \mik 2^{\beta^n}$
for every $n\in\nn$ with $n\meg 1$. We will need the following elementary facts.
\begin{fact} \label{f36}
For every vector strong subtree $\bfcs$ of $\bfct$ and every $n\in\nn$ with $n\meg 1$ we have $|\fan(\bfcs,n)|=\Theta_n$.
\end{fact}
\begin{fact} \label{newfact}
For every vector strong subtree $\bfcr$ of $\bfct$, every vector strong subtree $\bfcs$ of $\bfcr$ and every $n\in\nn$
with $n\meg 1$ there exists $k\in\nn$ with $k\meg n$ such that $\fan(\bfcs,n)\subseteq\fan(\bfcr,k)$.
\end{fact}
We define a sequence $(\theta_n)$ in $\rr$ by the rule $\theta_0=1$ and
\begin{equation} \label{e332}
\theta_n=\frac{\ee}{2b_W\Theta_n}
\end{equation}
for every $n\in\nn$ with $n\meg 1$. Notice that for every $n\in\nn$ we have
\begin{equation} \label{e333}
\theta_{n+1}\mik \theta_n.
\end{equation}

\subsection*{Step 2: a family $\{\mathcal{F}_n: n\meg 1\}$ of subsets of $\fan(\bfct)$}

Let $(\theta_n)$ be the sequence defined in Step 1. For every $n\in\nn$ with $n\meg 1$ we define a subset $\mathcal{F}_n$
of $\fan(\bfct)$ by the rule
\begin{eqnarray*}
\bfcf\in\mathcal{F}_n & \Leftrightarrow & \text{there exists a map } \phi:D\big(\bfcf(0)\big)\to\{0,...,b_W-1\} \\
& & \text{such that for every } w\in D\big(\bfcf(0)\big) \text{ if } p=\phi(w), \\
& & \text{then } \dens\Big( \bigcap_{\bft\in\otimes \bfcf(1)} D(\bft) \ \big| \ w^{\con_W}\!p \ \Big)\mik \theta_n.
\end{eqnarray*}
For every $\bfcf\in\mathcal{F}_n$ there exists a canonical map $\phi^n_{\bfcf}$ witnessing that $\bfcf$ belongs to
$\mathcal{F}_n$. It is defined by setting  $\phi^n_{\bfcf}(w)$ to be the least $p\in\{0,...,b_W-1\}$ for which the
above inequality is satisfied. We will call the map $\phi^n_{\bfcf}$ the \textit{witness} of $\bfcf$.

The next lemma reduces hypothesis (H) to certain properties of the sets in the family $\{\mathcal{F}_n:n\meg 1\}$.
\begin{lem} \label{l31newstaff}
Under hypothesis \emph{(H)}, for every $n\in\nn$ with $n\meg 1$ the following hold.
\begin{enumerate}
\item[(a)] We have $\mathcal{F}_{n+1}\subseteq\mathcal{F}_n$.
\item[(b)] For every vector strong subtree $\bfcs$ of $\bfct$ there exists $\bfcf\in\fan(\bfcs)\cap\mathcal{F}_n$
with $\bfcf(0)=\bfcs(0)$.
\end{enumerate}
\end{lem}
\begin{proof}
Part (a) follows by (\ref{e333}) and the relevant definitions. For part (b) we will argue by contradiction. So,
assume that there exist $n_0\in\nn$ with $n_0\meg 1$ and a vector strong subtree $\bfcs$ of $\bfct$ such that
for every $\bfcf\in\fan(\bfcs)$ with $\bfcf(0)=\bfcs(0)$ we have that $\bfcf\notin\mathcal{F}_{n_0}$. This implies
that for every $\bfcf\in\fan(\bfcs)$ with $\bfcf(0)=\bfcs(0)$ there exists $w_{\bfcf}\in D\big(\bfcs(0)\big)$
such that for every $p\in\{0,...,b_W-1\}$ we have
\[ \dens\Big( \bigcap_{\bfs\in\otimes \bfcf(1)} D(\bfs) \ \big| \ w_{\bfcf}^{\con_W}\!p \ \Big)\meg \theta_{n_0}. \]
The set $D\big(\bfcs(0)\big)$ is finite. Therefore, by Proposition \ref{pns3}, there exist a vector strong
subtree $\bfcr$ of $\bfcs$ with $\bfcr(0)=\bfcs(0)$ and $w_0\in D\big(\bfcs(0)\big)$ such that $w_\bfcf=w_0$
for every $\bfcf\in\fan(\bfcr)$ with $\bfcf(0)=\bfcr(0)$. It follows that the pair $(\bfcr,w_0)$ is strongly
$\theta_{n_0}$-correlated with respect to $D$ and this contradicts hypothesis (H). The proof of Lemma
\ref{l31newstaff} is completed.
\end{proof}

\subsection*{Step 3: control of vector fans with a fixed root}

Let $\bfcr$ be an arbitrary vector strong subtree of $\bfct$. Our goal in this step is to construct a vector strong
subtree $\bfcs$ of $\bfcr$ with the same root as $\bfcr$ such that for every vector fan $\bfcf$ of $\bfcs$ with
$\bfcf(0)=\bfcs(0)$ we have significant control over the quantity appearing in the left side of inequality (\ref{e311}).
Precisely, we will show the following.
\begin{lem} \label{l37}
Let $(\theta_n)$ be the sequence defined in Step 1. Also let $\bfcr$ be a vector strong subtree of $\bfct$ and set
$\bfr_0=\bfcr(0)$. Then there exist a vector strong subtree $\bfcs$ of $\bfcr$ with $\bfcs(0)=\bfr_0$ and a map
$\phi:D(\bfr_0)\to \{0,...,b_W-1\}$ such that the following is satisfied. For every $n\in\nn$ with $n\meg 1$,
every $\bfcf\in\fan(\bfcs,n)$ with $\bfcf(0)=\bfr_0$ and every $w\in D(\bfr_0)$ if $p=\phi(w)$, then
\[ \dens\Big( \bigcap_{\bfs\in\otimes \bfcf(1)} D(\bfs) \ \big| \ w^{\con_W}\!p \ \Big)\mik \theta_n. \]
\end{lem}
\begin{proof}
By part (b) of Lemma \ref{l31newstaff}, we may apply Proposition \ref{pns4} to the vector homogeneous tree $\bfct$,
the family $\{\mathcal{F}_n:n\meg 1\}$ and the vector strong subtree $\bfcr$ of $\bfct$. Therefore, there exists
a vector strong subtree $\bfcz$ of $\bfcr$ with $\bfcz(0)=\bfr_0$ such that for every $k\in\nn$ with $k\meg 1$
and every $\bfcf\in\fan(\bfcz,k)$ with $\bfcf(0)=\bfr_0$ we have that $\bfcf\in\mathcal{F}_k$. Let
$\phi^k_{\bfcf}:D(\bfr_0)\to\{0,...,b_W-1\}$ be the corresponding witness. The set $\{0,...,b_W-1\}^{D(\bfr_0)}$
is finite. By Proposition \ref{pns3}, there exist a vector strong subtree $\bfcs$ of $\bfcz$ with $\bfcs(0)=\bfr_0$
and a map $\phi:D(\bfr_0)\to\{0,...,b_W-1\}$ such that for every $\bfcf\in\fan(\bfcs)$ with $\bfcf(0)=\bfr_0$
if $k$ is the unique integer such that $\bfcf\in\fan(\bfcz,k)$, then $\phi^k_{\bfcf}=\phi$.
The vector tree $\bfcs$ and the map $\phi$ are as desired.

Indeed, let $n\in\nn$ with $n\meg 1$ and $\bfcf\in\fan(\bfcs,n)$ with $\bfcf(0)=\bfr_0$ be arbitrary. By
Fact \ref{newfact}, there exists $k\in\nn$ with $k\meg n$ such that $\bfcf\in\fan(\bfcz,k)$.
Let $w\in D(\bfr_0)$ be arbitrary. If $p=\phi(w)$, then $p=\phi^k_{\bfcf}(w)$. Hence,
\[ \dens\Big( \bigcap_{\bfs\in\otimes \bfcf(1)} D(\bfs) \ \big| \ w^{\con_W}\!p \ \Big)\mik
\theta_k \stackrel{(\ref{e333})}{\mik} \theta_n. \]
The proof of Lemma \ref{l37} is completed.
\end{proof}

\subsection*{Step 4: construction of an ``asymptotically sparse" vector tree}

In this step we will refine the construction presented in Step 3. Our goal is to construct an ``asymptotically sparse" vector
tree, i.e. a vector strong subtree $\bfcs$ of $\bfct$ for which we have control over the behavior of \textit{every} vector fan
of $\bfcs$. Specifically, we have the following.
\begin{lem} \label{l38}
Let $(\theta_n)$ be the sequence defined in Step 1. Then there exists a vector strong subtree $\bfcs$ of $\bfct$ with the following
property. For every $\bfs_0\in\otimes\bfcs$ there exists a map $\phi_{\bfs_0}: D(\bfs_0)\to \{0,...,b_W-1\}$ such that for every
$n\in\nn$ with $\ell_{\bfcs}(\bfs_0)<n$, every $\bfcf\in\fan(\bfcs,n)$ with $\bfcf(0)=\bfs_0$ and every $w\in D(\bfs_0)$
if $p=\phi_{\bfs_0}(w)$, then
\[ \dens\Big( \bigcap_{\bfs\in\otimes \bfcf(1)} D(\bfs) \ \big| \ w^{\con_W}\!p \ \Big)\mik \theta_n. \]
\end{lem}
\begin{proof}
Let us say that a vector strong subtree of $\bfct$ is in a \textit{good position} if it satisfies the conclusion
of Lemma \ref{l37}. That is, a vector strong subtree $\bfcz$ of $\bfct$ is in a good position if there exists a map 
$\phi:D\big(\bfcz(0)\big)\to\{0,...,b_W-1\}$ such that for every $k\in\nn$ with $k\meg 1$, every $\bfcf\in\fan(\bfcz,k)$
with $\bfcf(0)=\bfcz(0)$ and every $w\in D\big(\bfcz(0)\big)$ if $p=\phi(w)$, then
$\dens\big( \bigcap_{\bfz\in\otimes \bfcf(1)} D(\bfz) \ | \ w^{\con_W}\!p \big)\mik \theta_k$.

We notice two permanence properties of this notion. The first one is that it is \textit{hereditary} when passing
to vector subtrees. Precisely, if a vector strong subtree $\bfcz$ of $\bfct$ is in a good position and $\bfcz'$ is a vector
strong subtree of $\bfcz$ with $\bfcz'(0)=\bfcz(0)$, then $\bfcz'$ is also in a good position. This can be easily checked
arguing as in the proof of Lemma \ref{l37} and using the fact that the sequence $(\theta_n)$ is decreasing. The second
property is that the family of vector strong subtrees of $\bfct$ which are in a good position is \textit{dense}, i.e.
for every vector strong subtree $\bfcr$ of $\bfct$ there exists a vector strong subtree $\bfcz$ of $\bfcr$ with
$\bfcz(0)=\bfcr(0)$ such that $\bfcz$ is in a good position. This is, of course, the content of Lemma \ref{l37}.
Using these properties and a standard recursive construction, it is possible to find a vector strong subtree
$\bfcv$ of $\bfct$ such that for every $\bfv\in\otimes\bfcv$ the vector strong subtree $\suc_{\bfcv}(\bfv)$
of $\bfct$ is in a good position.

The desired vector tree $\bfcs$ will be an appropriately chosen vector strong subtree of $\bfcv$. Specifically,
let $(m_j)$ be a sequence in $\nn$ such that for every $j\in\nn$ with $j\meg 1$ we have $m_j\meg m_{j-1}+j$. We select
a vector strong subtree $\bfcs$ of $\bfcv$ such that $\otimes\bfcs(j)\subseteq\otimes\bfcv(m_j)$ for every $j\in\nn$.
We will show that $\bfcs$ is as desired. To this end, let $\bfs_0\in\otimes\bfcs$ be arbitrary and set $k=\ell_{\bfcs}(\bfs_0)$.
By the properties of $\bfcv$, the vector tree $\suc_{\bfcv}(\bfs_0)$ is in a good position. We fix a map $\phi_{\bfs_0}:
D(\bfs_0)\to \{0,...,b_W-1\}$ witnessing this fact. Let $n\in\nn$ with $k<n$ and $\bfcf\in\fan(\bfcs,n)$ with $\bfcf(0)=\bfs_0$.
Observe that $\bfs_0\in\otimes\bfcs(k)\subseteq\otimes\bfcv(m_k)$ and $\otimes\bfcf(1)\subseteq\otimes\bfcs(n)\subseteq\otimes\bfcv(m_n)$.
By the choice of the sequence $(m_j)$, there exists $l\in\nn$ with $l\meg n$ such that $\bfcf\in\fan\big(\suc_{\bfcv}(\bfs_0),l\big)$.
It follows that for every $w\in D(\bfs_0)$ if $p=\phi_{\bfs_0}(w)$, then
\[ \dens\Big( \bigcap_{\bfs\in\otimes \bfcf(1)} D(\bfs) \ \big| \ w^{\con_W}\!p \ \Big)\mik
\theta_l \stackrel{(\ref{e333})}{\mik} \theta_n.\]
The proof of Lemma \ref{l38} is completed.
\end{proof}

\subsection*{Step 5: fixing the ``direction"}

Let $\bfcs$ be the vector strong subtree of $\bfct$ obtained by Lemma \ref{l38}. For every $p\in \{0,...,b_W-1\}$
we define $C_p:\otimes\bfcs\to 2^W$ by the rule
\begin{equation} \label{e334}
C_p(\bfs)=\{w\in D(\bfs): \phi_{\bfs}(w)=p\}.
\end{equation}
\begin{lem} \label{l39}
There exist a vector strong subtree $\bfcz$ of $\bfcs$ and $p_0\in\{0,...,b_W-1\}$ such that for every $\bfz\in\otimes\bfcz$ we have
\[ \dens\big( C_{p_0}(\bfz)\big)\meg \ee/b_W. \]
\end{lem}
\begin{proof}
Let $\bfs\in\otimes\bfcs$ be arbitrary. Let $p_{\bfs}$ be the least $p\in\{0,...,b_W-1\}$ such that $\dens\big(C_p(\bfs)\big)\meg \ee/b_W$.
Notice that, by the classical pigeonhole principle, $p_{\bfs}$ is well-defined. By $\hl(d)$, there exist a vector strong subtree $\bfcz$
of $\bfcs$ and $p_0\in\{0,..., b_W-1\}$ such that $p_{\bfz}=p_0$ for every $\bfz\in\otimes\bfcz$. It is clear that $\bfcz$ and $p_0$ are
as desired.
\end{proof}

\subsection*{Step 6: properties of $C_{p_0}$}

In this step we will not construct something new but rather summarize what we have achieved so far. Let $\bfcz$ and $p_0$ be as in Lemma
\ref{l39}. Then the following are satisfied.
\begin{enumerate}
\item[($\mathcal{P}$1)] For every $\bfz\in\otimes\bfcz$ we have $C_{p_0}(\bfz)\subseteq D(\bfz)$.
\item[($\mathcal{P}$2)] The map $C_{p_0}:\otimes\bfcz\to 2^W$ is an $(\ee/b_W)$-dense level selection.
\item[($\mathcal{P}$3)] For every $k\in\nn$ with $k\meg 1$, every $\bfcf\in\fan(\bfcz,k)$ and every $w\in C_{p_0}(\bfz_0)$, where
$\bfz_0=\bfcf(0)$, we have $\dens\big( \bigcap_{\bfz\in\otimes \bfcf(1)} D(\bfz) \ | \ w^{\con_W}\!p_0\big)\mik \theta_k$.
\end{enumerate}
Property ($\mathcal{P}$1) follows immediately by (\ref{e334}). Property ($\mathcal{P}$2) is essentially the content of Lemma \ref{l39}.
Finally, property ($\mathcal{P}$3) follows by Lemma \ref{l38} and Fact \ref{newfact}.

\subsection*{Step 7: a sequence of ``forbidden" subsets of $W$}

Let $C_{p_0}:\otimes\bfcz\to 2^W$ be the $(\ee/b_W)$-dense level selection obtained in Step 5. Let $(l_k)$ be the strictly
increasing sequence in $\nn$ such that for every $k\in\nn$ and every $\bfz\in\otimes\bfcz(k)$ we have $C_{p_0}(\bfz)\subseteq W(l_k)$.

For every $k\in\nn$ with $k\meg 1$ we define a subset $G_k$ of $W(l_k)$ by the rule
\begin{eqnarray*}
w'\in G_k & \Leftrightarrow & \text{there exist } \bfcf\in \fan(\bfcz,k) \text{ and } w\in C_{p_0}(\bfz_0),
\text{ where } \bfz_0=\bfcf(0),\\
& & \text{such that } w'\in \bigcap_{\bfz\in\otimes \bfcf(1)} D(\bfz)\cap \suc_W(w^{\con_W}\!p_0).
\end{eqnarray*}
We view the sequence $(G_k)$ as a sequence of ``forbidden" subsets of $W$. Specifically, we will modify the dense level selection
$C_{p_0}$ in such a way that the range of the new one will be disjoint from every $G_k$. But in order to do so, we will need the
following estimate on the size of each $G_k$.
\begin{lem} \label{l310}
For every $k\in\nn$ with $k\meg 1$ we have
\[ \dens(G_k)\mik \frac{\ee}{2b_W}. \]
\end{lem}
\begin{proof}
For every $\bfcf\in\fan(\bfcz,k)$ let
\[ H_{\bfcf}=\bigcap_{\bfz\in\otimes \bfcf(1)} D(\bfz)\cap \suc_W\Big( \big\{ w^{\con_W}\!p_0: w\in C_{p_0}\big(\bfcf(0)\big)\big\}\Big)\]
and notice that
\[ G_k=\bigcup_{\bfcf\in\fan(\bfcz,k)} H_{\bfcf}. \]
Since $\bfcz$ is a vector strong subtree of the vector homogeneous tree $\bfct$, by Fact \ref{f36}, we have $|\fan(\bfcz,k)|=\Theta_k$.
Therefore, it is enough to show that for every $\bfcf\in\fan(\bfcz,k)$ we have $\dens(H_{\bfcf})\mik \ee/(2b_W\Theta_k)$.

To this end, let $\bfcf\in\fan(\bfcz,k)$ be arbitrary and set $\bfz_0=\bfcf(0)$. Also let $\lambda=\dens\big(C_{p_0}(\bfz_0)\big)/b_W$
and observe that $\lambda\mik 1$. The tree $W$ is homogeneous. Hence,
\begin{eqnarray*}
\dens(H_{\bfcf}) & \mik & \lambda \cdot \max \Big\{ \dens\Big( \bigcap_{\bfz\in\otimes \bfcf(1)} D(\bfz) \ \big| \ w^{\con_W}\!p_0 \Big)
: w\in C_{p_0}(\bfz_0) \Big\}\\
& \mik & \max \Big\{ \dens\Big( \bigcap_{\bfz\in\otimes \bfcf(1)} D(\bfz) \ \big| \ w^{\con_W}\!p_0\Big) : w\in C_{p_0}(\bfz_0) \Big\}\\
& \stackrel{(\mathcal{P}3)}{\mik} & \theta_k \stackrel{(\ref{e332})}{=} \frac{\ee}{2b_W\Theta_k}.
\end{eqnarray*}
The proof of Lemma \ref{l310} is completed.
\end{proof}

\subsection*{Step 8: definition of the dense level selection $B$}

Let $C_{p_0}:\otimes\bfcz\to 2^W$ be the $(\ee/b_W)$-dense level selection obtained in Step 5.
Also let $(G_k)$ be the sequence of subsets of $W$ defined in Step 7.

We define $B:\otimes\bfcz\to 2^W$ as follows. First we set $B\big(\bfcz(0)\big)=C_{p_0}\big(\bfcz(0)\big)$.
If $\bfz\in\otimes\bfcz(k)$ for some $k\in\nn$ with $k\meg 1$, then we set
\begin{equation} \label{e336}
B(\bfz)=C_{p_0}(\bfz)\setminus G_k.
\end{equation}
We summarize, below, the main properties of the map $B$.
\begin{enumerate}
\item[($\mathcal{P}$4)] For every $\bfz\in\otimes\bfcz$ we have $B(\bfz)\subseteq C_{p_0}(\bfz)\subseteq D(\bfz)$.
\item[($\mathcal{P}$5)] The map $B:\otimes\bfcz\to 2^W$ is an $(\ee/2b_W)$-dense level selection.
\item[($\mathcal{P}$6)] For every $\bfcf\in\fan(\bfcz)$ and every $w\in B(\bfz_0)$, where $\bfz_0=\bfcf(0)$, we have
\[ \bigcap_{\bfz\in\otimes\bfcf(1)} B(\bfz)\cap \suc_W(w^{\con_W}\!p_0)=\varnothing.\]
\end{enumerate}
Property ($\mathcal{P}$4) follows by property ($\mathcal{P}$1) isolated in Step 6 and (\ref{e336}). Property ($\mathcal{P}$5) follows
by property ($\mathcal{P}$2) and Lemma \ref{l310}. To see that property ($\mathcal{P}$6) is satisfied, let $\bfcf\in\fan(\bfcz)$ and
$w\in B(\bfz_0)$, where $\bfz_0=\bfcf(0)$. We set
\[ A=\bigcap_{\bfz\in\otimes\bfcf(1)} B(\bfz)\cap \suc_W(w^{\con_W}\!p_0).\]
Let $k$ be the unique integer such that $\bfcf\in\fan(\bfcz,k)$. By property ($\mathcal{P}$4) and the definition of $G_k$,
we see that $A\subseteq G_k$. We fix $\bfz'\in\otimes\bfcf(1)$ and we notice that $A\subseteq B(\bfz')$. Since $\bfz'\in\otimes\bfcz(k)$,
by the previous inclusions and the definition of the dense level selection $B$, we conclude that $A\subseteq G_k\cap B(\bfz')=\varnothing$,
as desired.

\subsection*{Step 9: getting the contradiction}

We are finally in a position to derive the contradiction. Indeed, by property ($\mathcal{P}$5), the map $B:\otimes\bfcz\to 2^W$
is an $(\ee/2b_W)$-dense level selection. Moreover, by our assumptions, we have that $\dhl(d)$ holds. Therefore, by Proposition
\ref{p35} applied to ``$B:\otimes\bfcz\to 2^W$" and ``$p_0$", there must exist $\bfcf\in\fan(\bfcz)$ and $w\in B(\bfz_0)$,
where $\bfz_0=\bfcf(0)$, such that
\[ \bigcap_{\bfz\in\otimes\bfcf(1)} B(\bfz)\cap \suc_W(w^{\con_W}\!p_0)\neq\varnothing. \]
This contradicts property ($\mathcal{P}$6). The proof of Theorem \ref{t33} is thus completed.

\subsection{Comments}

We recall that a vector tree $\bfct=(T_1,...,T_d)$ is said to be \textit{boundedly branching} if for every $i\in\{1,...,d\}$ there exists
$b_i\in\nn$ with $b_i\meg 1$ such that every $t\in T_i$ has at most $b_i$ immediate successors. By Theorem \ref{t11}, for every vector
boundedly branching, pruned tree $\bfct=(T_1,...,T_d)$ there exist a vector strong subtree $\bfcs=(S_1,...,S_d)$ of $\bfct$ and $b_1,...,b_d\in\nn$,
with $b_i\meg 1$ for every $i\in\{1,...,d\}$, such that every $s\in S_i$ has exactly $b_i$ immediate successors in $S_i$ for every $i\in\{1,...,d\}$.
Therefore, Theorem \ref{t33} is also valid for nontrivial vector boundedly branching trees simply by reducing the general case to the case of
vector homogeneous trees.

The next natural class of vector trees for which Theorem \ref{t33} could be possibly true is that of vector \textit{quasi-homogeneous}
trees: a vector tree $\bfct=(T_1,...,T_d)$ is said to be quasi-homogeneous if for every $i\in\{1,...,d\}$ the number of immediate successors
of a node in $T_i$ depends only on its length. We point out that all arguments in this section (as well as the recursive construction
presented in \S 5) can be easily adapted to treat vector quasi-homogeneous trees \textit{except} Fact \ref{f36}. Indeed, by Fact \ref{f36},
we have an a priori estimate for the cardinality of the set $\fan(\bfcs,k)$ for every $k\in\nn$ and every vector strong subtree $\bfcs$
of a vector homogeneous tree $\bfct$. If the vector tree $\bfct$ is quasi-homogeneous but not boundedly branching, then no estimate can be
obtained. As is shown in \cite[Theorem 2.5]{BV} this obstacle is a necessity rather than a coincidence.

Finally we remark that, using essentially the same arguments as in the proof of Theorem \ref{t33}, one can show that there
exist a vector strong subtree $\bfcr$ of $\bfct$ and a constant $0<\theta\mik 1$ such that for ``almost all" nodes $w$
in $D\big(\bfcr(0)\big)$ the pair $(\bfcr,w)$ is strongly $\theta$-correlated with respect to $D$. Precisely, we have the following.
\begin{thm} \label{newthm}
Let $d\meg 1$ and assume that $\mathrm{DHL}(d)$ holds. Also let $\bfct=(T_1,...,T_d)$ be a vector homogeneous tree,
$W$ a homogeneous tree, $0<\ee\mik 1$ and $D:\otimes\bfct\to 2^W$ an $\ee$-dense level selection. Then for every
$0<\delta<\ee$ there exist a vector strong subtree $\bfcr$ of $\bfct$ and a constant $0<\theta\mik 1$ such that,
setting $\bfr_0=\bfcr(0)$ and
\[ G=\big\{ w\in D(\bfr_0): \text{ the pair } (\bfcr,w) \text{ is strongly $\theta$-correlated with respect to $D$} \big\},\]
we have $\dens(G)\meg \dens\big(D(\bfr_0)\big)-\delta$.
\end{thm}


\section{Proof of Theorem \ref{t12}}

As we have already mentioned in the introduction, the proof of Theorem \ref{t12} proceeds by induction. The case ``$d=1$"
is the content of \cite[Theorem 2.3]{BV}. So, assume that we have proven $\dhl(d)$ for some $d\in\nn$ with $d\meg 1$ and that
we are given a vector homogeneous tree $(T_1,...,T_d,W)$, a constant $0<\ee\mik 1$, a subset $D$ of the level product of
$(T_1,...,T_d,W)$ and an infinite subset $L$ of $\nn$ such that
\begin{equation} \label{e41}
|D\cap \big(T_1(n)\times ... \times T_d(n)\times W(n)\big)|\meg \ee |\big(T_1(n)\times ... \times T_d(n)\times W(n)\big)|
\end{equation}
for every $n\in L$. Our goal is to find a vector strong subtree $(Z_1,..., Z_d,V)$ of $(T_1,..., T_d,W)$ whose level product is
contained in $D$. This will be done in several steps. We set $\bfct=(T_1,...,T_d)$. For notational simplicity, for every
$i\in\{1,...,d\}$ by $b_i$ we shall denote the branching number of the homogeneous tree $T_i$. The vector homogeneous tree
$\bfcb_{\bfct}=(b_1^{<\nn},...,b_d^{<\nn})$ will be denoted simply by $\bfcb$. Notice that if $\bfp\in\otimes\bfcb(1)$, then $\bfp$
is a finite sequence $(p_1,...,p_d)$ with $p_i\in \{0,...,b_i-1\}$ for every $i\in\{1,...,d\}$. By $\mhden$ we shall denote the unique
finite sequence in $\otimes\bfcb(1)$ having zero entries. For every $n\in\nn$, every $\bfu=(u_1,...,u_d)\in\otimes\bfcb(n)$ and every
$\bfp=(p_1,...,p_d)\in\otimes\bfcb(1)$ we set $\bfu^{\con}\bfp=(u_1^{\con}p_1,...,u_d^{\con}p_d)\in\otimes\bfcb(n+1)$.

\subsection*{Step 1: obtaining a dense level selection}

For every $n\in L$ let $C_n$ be the subset of $\otimes\bfct(n)$ defined by the rule
\[ \bft\in C_n \Leftrightarrow |\{w\in W(n): (\bft,w)\in D\}|\meg \ee/2 |W(n)|.\]
Using (\ref{e41}) and arguing as in the proof of Claim \ref{fubini}, we get the following.
\begin{fact} \label{f41}
For every $n\in L$ we have $|C_n|\meg \ee/2 |\otimes \bfct(n)|$.
\end{fact}
We set $C=\bigcup_{n\in L} C_n$. By Fact \ref{f41}, we see that
\[ \limsup_{n\to\infty} \frac{|C\cap \otimes\bfct(n)|}{|\otimes\bfct(n)|} =
\limsup_{n\in L} \frac{|C_n\cap \otimes\bfct(n)|}{|\otimes\bfct(n)|}\meg \ee/2 > 0.\]
Since $\dhl(d)$ holds, there exists a vector strong subtree $\bfcs$ of $\bfct$ such that $\otimes\bfcs\subseteq C$.
It follows that the section map
\[ \otimes\bfcs \ni \bfs\mapsto \{w\in W: (\bfs,w)\in D\}\in 2^W \]
is an $(\ee/2)$-dense level selection. It will be denoted by $D:\otimes\bfcs\to 2^W$.

\subsection*{Step 2: defining certain vector fans}

Let $\bfcs=(S_1,...,S_d)$ be the vector strong subtree of $\bfct$ obtained in Step 1. Also let $\bfcr$ be an arbitrary vector strong
subtree of $\bfcs$. In this step, we will introduce a method to obtain vector fans of $\bfcr$ from certain elements of $\otimes\bfcr$.
The method is based on the notion of a vector canonical isomorphism described in \S 2.5. The resulting vector fans will be used in the
next step.

We will describe, first, the one-dimensional case in abstract setting. So, let $Z$ be a homogeneous tree and set $z_0=Z(0)$. 
For every $p\in\{0,...,b_Z-1\}$ let
\begin{equation} \label{e42}
Z[p]= \suc_{Z}( z_0^{\con_Z}\!p).
\end{equation}
It is clear that $Z[p]$ is a strong subtree of $Z$, and so, it is homogeneous with branching number $b_Z$. This observation
permits us to consider the canonical isomorphism $\ci\big(Z[0],Z[p]\big)$ between $Z[0]$ and $Z[p]$ for every $p\in\{0,...,b_Z-1\}$.
Now, for every $z\in Z[0]$ we set
\begin{equation} \label{e43}
F_{z,Z}=\{z_0\}\cup \Big\{ \ci\big(Z[0],Z[p]\big)(z): p\in\{0,...,b_Z-1\}\Big\}.
\end{equation}
Notice that $F_{z,Z}\in\fan(Z)$. The fan $F_{z,Z}$ will be called a \textit{$(z,Z)$-directed fan}. We point that not
every fan of $Z$ is $(z,Z)$-directed for some $z\in Z[0]$. Actually, the set of all $(z,Z)$-directed fans is a
rather ``thin" subset of $\fan(Z)$.

After this preliminary discussion, we are ready to introduce the vector fans we mentioned above. Specifically, let $\bfcr=(R_1,...,R_d)$
be an arbitrary vector strong subtree of $\bfcs$. For every $\bfp=(p_1,...,p_d)\in \otimes\bfcb(1)$ we set
\begin{equation} \label{e44}
\bfcr[\bfp]=\big( R_1[p_1], ..., R_d[p_d]\big)
\end{equation}
and we notice that $\bfcr[\bfp]$ is a vector strong subtree of $\bfcr$. Again we emphasize that this observation permits us
to consider that vector canonical isomorphism $\bfci\big(\bfcr[\mhden],\bfcr[\bfp]\big)$ between $\bfcr[\mhden]$ and $\bfcr[\bfp]$
for every $\bfp\in\otimes\bfcb(1)$. Observe that $\bfci\big(\bfcr[\mhden],\bfcr[\mhden]\big)$ is the identity map on $\bfcr[\mhden]$.
For every $\bfr=(r_1,...,r_d)\in\otimes\bfcr[\mhden]$ we define
\begin{equation} \label{e45}
\bfcf_{\bfr,\bfcr}=(F_{r_1,R_1}, ..., F_{r_d,R_d}).
\end{equation}
Notice that $\bfcf_{\bfr,\bfcr}$ is well-defined since $r_i\in R_i[0]$ for every $i\in\{1,...,d\}$. Also observe that
$\bfcf_{\bfr,\bfcr}\in\fan(\bfcr)$ and $\bfcf_{\bfr,\bfcr}(0)=\bfcr(0)$. The vector fan $F_{\bfr,\bfcr}$ will be called
an \textit{$(\bfr,\bfcr)$-directed vector fan}. We isolate, for future use, the following representation of the set
$\otimes\bfcf_{\bfr,\bfcr}(1)$. It is a direct consequence of the relevant definitions.
\begin{fact} \label{f42}
For every vector strong subtree $\bfcr$ of $\bfcs$ and every $\bfr\in\bfcr[\mhden]$ we have
\[ \otimes\bfcf_{\bfr,\bfcr}(1)=\Big\{ \bfci\big(\bfcr[\mhden],\bfcr[\bfp]\big)(\bfr): \bfp\in\otimes\bfcb(1)\Big\}. \]
In particular, we have $\bfr\in \otimes\bfcf_{\bfr,\bfcr}(1)$.
\end{fact}

\subsection*{Step 3: a recursive construction}

This is the main step of the proof. Let $D:\otimes\bfcs\to 2^W$ be the $(\ee/2)$-dense level selection obtained in Step 1.
Recursively, we shall construct
\begin{enumerate}
\item[(a)] two sequences $(\bfcs_n)$ and $(\bfcr_n)$ of vector strong subtrees of $\bfcs$,
\item[(b)] two sequences $(\ee_n)$ and $(\theta_n)$ of reals in $(0,1]$,
\item[(c)] a strictly increasing sequence $(l_n)$ in $\nn$,
\item[(d)] for every $n\in\nn$ a map $D_n:\otimes\bfcs_n\to 2^W$ and
\item[(e)] a family $\{w_v:v\in b_W^{<\nn}\}$ in $W$
\end{enumerate}
such that for every $n\in\nn$ the following conditions are satisfied.
\begin{enumerate}
\item[(C1)] $\bfcr_n$ is a vector strong subtree of $\bfcs_n$.
\item[(C2)] $\bfcs_{n+1}=\bfcr_n[\mhden]$.
\item[(C3)] For every $v\in b_W^n$ we have $\ell_W(w_v)=l_n$.
\item[(C4)] For every $v\in b_W^n$ and $p\in\{0,...,b_W-1\}$ we have $w_{v^{\con}\!p}\in \suc_W(w_v^{\con_W}\!p)$.
\item[(C5)] The map $D_n:\otimes\bfcs_n\to 2^W$ is an $\ee_n$-dense level selection.
\item[(C6)] For every $\bfs\in\otimes\bfcs_n$ we have $D_n(\bfs)\subseteq D_0(\bfs)=D(\bfs)$.
\item[(C7)] For every $v\in b_W^n$ the pair $(\bfcr_n,w_v)$ is strongly $\theta_n$-correlated with respect to the dense
level selection $D_n$.
\item[(C8)] For every $\bfr\in\otimes\bfcr_n[\mhden]$ we have
\begin{equation} \label{e46}
D_{n+1}(\bfr)=\bigcap_{\bfp\in\otimes\bfcb(1)} D_n\Big( \bfci\big(\bfcr_n[\mhden],\bfcr_n[\bfp]\big)(\bfr)\Big).
\end{equation}
\end{enumerate}
We proceed to the construction. For $n=0$ we set ``$\bfcs_0=\bfcs$", ``$\ee_0=\ee/2$" and ``$D_0=D$" and we notice that with
these choices conditions (C5) and (C6) are satisfied. Recall that we have already proven $\dhl(d)$. Therefore, by Theorem
\ref{t33} applied to the $\ee_0$-dense level selection $D_0:\otimes\bfcs_0\to 2^W$, there exist a vector strong subtree $\bfcr$
of $\bfcs_0$, a node $w\in W$ and a constant $0<\theta\mik 1$ such that the pair $(\bfcr,w)$ is strongly $\theta$-correlated with
respect to $D_0$. We set ``$\bfcr_0=\bfcr$", ``$\theta_0=\theta$", ``$w_{\varnothing}=w$" and ``$l_0=\ell_W(w)$" and we observe that
with these choices conditions (C1), (C3) and (C7) are satisfied. Since conditions (C2), (C4) and (C8) are meaningless for $n=0$,
the first step of the recursive construction is completed.

Let $n\in\nn$ and assume that the construction has been carried out up to $n$. We set ``$\bfcs_{n+1}=\bfcr_n[\mhden]$"
and we notice that condition (C2) is satisfied. Let $\bfr\in\otimes\bfcs_{n+1}$ be arbitrary and consider the
$(\bfr,\bfcr_n)$-directed fan $\bfcf_{\bfr,\bfcr_n}$ described in Step 2. We define $D_{n+1}:\otimes\bfcs_{n+1}\to 2^W$
by the rule
\[ D_{n+1}(\bfr)=\bigcap_{\bfs\in\otimes\bfcf_{\bfr,\bfcr_n}(1)} D_n(\bfs). \]
By Fact \ref{f42} and our inductive assumptions, we see that conditions (C6) and (C8) are satisfied. We set
``$\ee_{n+1}=\theta_n b_W^{n-l_n}$" and we claim that with this choice condition (C5) is satisfied.
To this end, it is enough to show that $\dens\big(D_{n+1}(\bfr)\big)\meg \ee_{n+1}$ for every $\bfr\in\otimes\bfcs_{n+1}$.
So, let $\bfr\in\otimes\bfcs_{n+1}$ be arbitrary. By our inductive assumptions, the pair $(\bfcr_n,w_v)$ is strongly
$\theta_n$-correlated with respect to $D_n$ for every $v\in b_W^n$. Recall that $\ell_W(w_v)=l_n$ and
$\bfcf_{\bfr,\bfcr_n}\in\fan(\bfcr_n)$. Since the tree $W$ is homogeneous we get that
\begin{eqnarray*}
\dens\big(D_{n+1}(\bfr)\big) & \meg & \frac{1}{b_W^{l_n+1}} \sum_{v\in b_W^n} \sum_{p=0}^{b_W-1}
\dens\big( D_{n+1}(\bfr) \ | \ w_v^{\con_W}\!p \big) \\
& = & \frac{1}{b_W^{l_n+1}} \sum_{v\in b_W^n} \sum_{p=0}^{b_W-1}
\dens\Big( \bigcap_{\bfs\in\otimes\bfcf_{\bfr,\bfcr_n}(1)} D_n(\bfs) \ \big| \ w_v^{\con_W}\!p \ \Big)\\
& \meg & \frac{b_W^{n+1}}{b_W^{l_n+1}} \cdot \theta_n = \ee_{n+1}.
\end{eqnarray*}
This shows that $D_{n+1}$ is an $\ee_{n+1}$-dense level selection.

Now for every $v\in b_W^n$ and every $p\in\{0,...,b_W-1\}$ we define $D_{v,p}:\otimes\bfcs_{n+1}\to 2^W$ by the rule
$D_{v,p}(\bfr)=D_{n+1}(\bfr)\cap \suc_W(w_v^{\con_W}\!p)$. Arguing as above, it is easy to check that $D_{v,p}$ is a
$\delta_n$-dense level selection where $\delta_n=\theta_n/b_W^{l_n+1}$. Again we emphasize that we have already proven
$\dhl(d)$. Therefore, by repeated applications of Corollary \ref{c34}, we may find a vector strong subtree $\bfcr$ of
$\bfcs_{n+1}$ and for every $v\in b_W^n$ and every $p\in\{0,...,b_W-1\}$ a node $w_{v,p}\in W$ and a constant
$0<\theta_{v,p}\mik 1$ such that the pair $(\bfcr,w_{v,p})$ is strongly $\theta_{v,p}$-correlated with respect
to $D_{v,p}$. We set ``$\bfcr_{n+1}=\bfcr$", ``$\theta_{n+1}=\min\big\{\theta_{v,p}:v\in b_W^n \text{ and }
p\in\{0,...,b_W-1\}\big\}$" and ``$w_{v^{\con}\!p}=w_{v,p}$" for every $v\in b_W^n$ and every $p\in\{0,...,b_W-1\}$.
Notice that with these choices conditions (C1), (C4) and (C7) are satisfied. Let $\bfr_{n+1}$ be the root of $\bfcr_{n+1}$.
Since the pair $(\bfcr_{n+1},w_{v^{\con}\!p})$ is strongly $\theta_{n+1}$-correlated with respect to $D_{v,p}$ we see
that $w_{v^{\con}\!p}\in D_{v,p}(\bfr_{n+1}) \subseteq D_{n+1}(\bfr_{n+1})\subseteq D(\bfr_{n+1})$. Hence, there exists $l\in\nn$
with $l>l_n$ such that $\ell_W(w_{v^{\con}\!p})=l$ for every $v\in b_W^n$ and every $p\in\{0,...,b_W-1\}$. We set
``$l_{n+1}=l$" and we observe that the last condition, condition (C3), is also satisfied. The recursive construction
is completed.

\subsection*{Step 4: a family of vector canonical embeddings}

Let $(\bfcs_n)$ and $(\bfcr_n)$ be the sequences of vector strong subtrees of $\bfcs$ obtained in Step 3. Also let $(D_n)$
be the corresponding sequence of dense level selections. Recall that $\bfcs_{n+1}=\bfcr_n[\mhden]$. For every $n\in\nn$
we write $\bfcs_n=(S^n_1,...,S^n_d)$ and $\bfcr_n=(R^n_1,...,R^n_d)$. Our goal in this step is to define a family
$\{H_{\bfu}:\bfu\in\otimes\bfcb\}$ of vector canonical embeddings. It will be used to ``unravel" the recursive
definition of the sequence $(D_n)$ and relate each $D_n$ with the $(\ee/2)$-dense level selection $D:\otimes\bfcs\to 2^W$
obtained in Step 1. This is the content of Fact \ref{f43} below.

To this end, we will describe first how these embeddings are acting in each coordinate. So, fix $i\in\{1,...,d\}$. Recursively,
for every $n\in\nn$ and every $u\in b_i^n$ we define a map $h^u_i:S^n_i\to T_i$ as follows. For $u=\varnothing$ let
$h^{\varnothing}_i:S^0_i\to T_i$ be the identity. Let $n\in\nn$ and $u\in b_i^n$ and assume that the
map $h^u_i:S^n_i\to T_i$ has been defined. For every $p\in\{0,...,b_i-1\}$ we set
\begin{equation} \label{e47}
h^{u^{\con}\!p}_i= h^u_i \circ \ci\big(S^{n+1}_i,R^n_i[p]\big)=h^u_i \circ \ci\big(R^n_i[0],R^n_i[p]\big)
\end{equation}
where $\ci\big(S^{n+1}_i,R^n_i[p]\big)$ is the canonical isomorphism between the homogeneous trees $S^{n+1}_i$ and $R^n_i[p]$.
Inductively, it is easy to verify the following properties guaranteed by the above construction.
\begin{enumerate}
\item[(P1)] For every $u\in b_i^{<\nn}$ the map $h^u_i$ is a canonical embedding.
\item[(P2)] For every $n\in\nn$, every $u\in b_i^n$ and every $s\in S^n_i$ we have $\ell_{T_i}\big(h^u_i(s)\big)=\ell_{T_i}(s)$.
\end{enumerate}

We are ready to introduce the desired family $\{H_\bfu:\bfu\in\otimes\bfcb\}$ of vector canonical embeddings.
So, let $n\in\nn$ and $\bfu=(u_1,...,u_d)\in\otimes\bfcb(n)$ be arbitrary. We define $H_{\bfu}:\otimes\bfcs_n\to\otimes\bfct$
by the rule
\begin{equation} \label{e48}
H_\bfu\big((s_1,...,s_d)\big)=\big( h^{u_1}_1(s_1), ..., h^{u_d}_d(s_d)\big).
\end{equation}
By properties (P1) and (P2), it is clear that $H_\bfu$ is a well-defined vector canonical embedding. We will need a formula
satisfied by these maps which follows by identities (\ref{e44}), (\ref{e47}) and (\ref{e48}). Specifically, for every $n\in\nn$,
every $\bfu=(u_1,...,u_d)\in\otimes\bfcb(n)$ and every $\bfp=(p_1,...,p_d)\in\otimes\bfcb(1)$ it holds that
\begin{equation} \label{e49}
H_{\bfu^{\con}\bfp}= H_\bfu \circ \bfci\big(\bfcr_n[\mhden],\bfcr_n[\bfp]\big)
\end{equation}
where $\bfci\big(\bfcr_n[\mhden],\bfcr_n[\bfp]\big)$ is the vector canonical isomorphism between $\bfcr_n[\mhden]$ and
$\bfcr_n[\bfp]$. We will also need the following.
\begin{fact} \label{f43}
For every $n\in\nn$ and every $\bfs\in\otimes\bfcs_n$ we have
\[ D_n(\bfs)=\bigcap_{\bfu\in\otimes\bfcb(n)} D\big( H_\bfu(\bfs)\big). \]
\end{fact}
\begin{proof}
The proof proceeds by induction on $n$. For $n=0$ the desired identity follows immediately by condition (C6) in Step 3 and the fact
that $H_{\bfcb(0)}$ is the identity map on $\otimes\bfcs_0$. Assume that the result has been proved for some $n\in\nn$.
Let $\bfs\in\otimes\bfcs_{n+1}$ be arbitrary. Recall that $\bfcs_{n+1}=\bfcr_n[\mhden]$ and that $\bfcr_n$
is a vector strong subtree of $\bfcs_n$. Therefore,
\begin{eqnarray*}
D_{n+1}(\bfs) & \stackrel{(\ref{e46})}{=} & \bigcap_{\bfp\in\otimes\bfcb(1)} D_n\Big( \bfci\big(\bfcr_n[\mhden],\bfcr_n[\bfp]\big)(\bfs)\Big)\\
& = & \bigcap_{\bfp\in\otimes\bfcb(1)} \bigcap_{\bfu\in\otimes\bfcb(n)} D\Big( H_\bfu\Big( \bfci\big(\bfcr_n[\mhden],\bfcr_n[\bfp]\big) (\bfs)\Big)\Big)\\
& \stackrel{(\ref{e49})}{=} & \bigcap_{\bfp\in\otimes\bfcb(1)} \bigcap_{\bfu\in\otimes\bfcb(n)}
D\big( H_{\bfu^{\con}\bfp}(\bfs)\big) =\bigcap_{\bfu'\in\otimes\bfcb(n+1)} D\big( H_{\bfu'}(\bfs)\big)
\end{eqnarray*}
where the second equality follows by our inductive assumption. The proof of Fact \ref{f43} is completed.
\end{proof}

\subsection*{Step 5: an infinite chain of $(T_1,...,T_d)$ and an ``unfolding" argument}

Let $(l_n)$ be the strictly increasing sequence in $\nn$ and $(\bfcr_n)$ the sequence of vector strong subtrees
of $\bfcs$ obtained in Step 3. For every $n\in\nn$ we set
\begin{equation} \label{e410}
\bfr_{n}=\bfcr_n(0)
\end{equation}
and we write $\bfr_n=(r^n_1,...,r^n_d)$. Recall that a subset $C$ of a tree $(T,<)$ is said to be a \textit{chain} if for every
$s,t\in C$ we have that either $t\mik s$ or $s\mik t$.
\begin{lem} \label{l44}
For every $i\in\{1,...,d\}$ the family $\{r^n_i:n\in\nn\}$ is an infinite chain of the tree $T_i$. Moreover, for every $n\in\nn$
we have $\ell_{T_i}(r^n_i)=l_n$.
\end{lem}
\begin{proof}
Let $n\in\nn$ be arbitrary. As in Step 4, we write $\bfcr_n=(R^n_1,...,R^n_d)$. By conditions (C1) and (C2) in Step 3,
we see that $r^{n+1}_i\in\suc_{R^n_i}(r^n_i)\subseteq \suc_{T_i}(r^n_i)$. This shows that the family $\{r^n_i:n\in\nn\}$
is an infinite chain of $T_i$. Also notice that, by conditions (C6) and (C7), we have $(r^n_1,...,r^n_d,w_v)\in D$ for every
$v\in b_W^n$. Invoking condition (C3), we conclude that $\ell_{T_i}(r^n_i)=l_n$. The proof of Lemma \ref{l44} is completed.
\end{proof}
Let $i\in\{1,...,d\}$ and consider the family $\{h^u_i:u\in b_i^{<\nn}\}$ of canonical embeddings defined in Step 4.
We define a map $\Phi_i:b_i^{<\nn}\to T_i$ as follows. For every $n\in\nn$ and every $u\in b_i^n$ we set
\begin{equation} \label{e411}
\Phi_i(u)= h^u_i( r^n_i).
\end{equation}
That is, for every $n\in\nn$ the family $\{\Phi_i(u):u\in b_i^n\}$ is the ``orbit" of the node $r^n_i$ under the family
of maps $\{h^u_i:u\in b_i^n\}$.
\begin{lem} \label{l45}
For every $i\in\{1,...,d\}$ the map $\Phi_i:b_i^{<\nn}\to T_i$ is a canonical embedding. Moreover, for every $n\in\nn$
and every $u\in b_i^n$ we have
\begin{equation} \label{e412}
\ell_{T_i}\big(\Phi_i(u)\big)=l_n.
\end{equation}
\end{lem}
\begin{proof}
Let $i\in\{1,...,d\}$ be arbitrary. First notice that, by property (P2) in Step 4 and Lemma \ref{l44}, condition (a) in \S 2.5
and equality (\ref{e412}) are both satisfied. To show that condition (b) in \S 2.5 is satisfied we need to prove that for every $n\in\nn$, 
every $u\in b_i^n$ and every $p\in\{0,...,b_i-1\}$ we have that $\Phi_i(u^{\con}p)\in\suc_{T_i}\big(\Phi_i(u)^{\con_{T_i}}p\big)$. To this
end let
\[ w=r^n_i \ \text{ and } \ r=\ci\big(R^n_i[0],R^n_i[p]\big)(r^{n+1}_i). \]
By (\ref{e47}) and (\ref{e411}), we see that $\Phi_i(u)=h^u_i(w)$ and $\Phi_i(u^{\con}p)=h^u_i(r)$. Also observe that
$r\in\suc_{T_i}(w^{\con_{T_i}}p)$. By property (P1) in Step 4, the map $h^u_i$ is a canonical embedding. Therefore,
$h^u_i(r)\in\suc_{T_i}\big(h^u_i(w)^{\con_{T_i}}p\big)$. The proof of Lemma \ref{l45} is completed.
\end{proof}

\subsection*{Step 6: the end of the proof}

For every $i\in\{1,...,d\}$ we set
\begin{equation} \label{e413}
Z_i=\{\Phi_i(u):u\in b_i^{<\nn}\}
\end{equation}
where $\Phi_i$ is the canonical embedding defined in (\ref{e411}). Also, we set
\begin{equation} \label{e414}
V=\{w_v:v\in b_W^{<\nn}\}
\end{equation}
where $\{w_v:v\in b_W^{<\nn}\}$ is the family obtained in part (e) of the construction presented in Step 3.
By conditions (C3) and (C4), we see that $V$ is a strong subtree of $W$ and $L_W(V)=\{l_n:n\in\nn\}$. Moreover, by Lemma
\ref{l45}, $Z_i$ is a strong subtree of $T_i$ and $L_{T_i}(Z_i)=\{l_n:n\in\nn\}$ for every $i\in\{1,...,d\}$.
It follows that $(Z_1,...,Z_d,V)$ is a vector strong subtree of $(T_1,...,T_d,W)$. The proof will be completed once we show
that the level product of $(Z_1,...,Z_d,V)$ is contained in $D$.

So, let $(z_1,...,z_d,v)$ be an arbitrary element of the level product of $(Z_1,...,Z_d,V)$. There exist
$n\in\nn$, $\bfu_0=(u_1,...,u_d)\in\otimes\bfcb(n)$ and $v_0\in b_W^n$ such that $v=w_{v_0}$ and
$z_i=\Phi_i(u_i)$ for every $i\in\{1,...,d\}$. Notice that
\[ \big(\Phi_1(u_1),...,\Phi_d(u_d)\big) \stackrel{(\ref{e411})}{=}
\big(h^{u_1}_1(r^n_1),...,h^{u_d}_d(r^n_d)\big) \stackrel{(\ref{e48})}{=}
H_{\bfu_0}(\bfr_n) \stackrel{(\ref{e410})}{=} H_{\bfu_0}\big(\bfcr_n(0)\big).\]
By condition (C7), the pair $(\bfcr_n,w_{v_0})$ is strongly $\theta_n$-correlated with respect to $D_n$.
Therefore, $w_{v_0}\in D_n\big(\bfcr_n(0)\big)$. By Fact \ref{f43}, we get that
\[ w_{v_0}\in D_n\big(\bfcr_n(0)\big)\subseteq D\big(H_{\bfu_0}\big(\bfcr_n(0)\big)\big).\]
Summing up, we conclude that
\[ (z_1,...,z_d,v)=\big(\Phi_1(u_1),...,\Phi_d(u_d), w_{v_0}\big)\in D.\]
The proof of Theorem \ref{t12} is thus completed.


\section{Comments}

Using a standard compactness argument we get the following finite version of Theorem \ref{t12}.
\begin{thm} \label{t51}
For every integer $d\meg 1$, every $b_1,...,b_d\in\nn$ with $b_i\meg 2$ for all $i\in\{1,...,d\}$, every integer $k\meg 1$, every real $0<\ee\mik 1$
and every infinite subset $M=\{m_0<m_1<...\}$ of $\nn$ there exists an integer $N$ with the following property. If $\bfct=(T_1,...,T_d)$
is a vector homogeneous tree with $b_{\bfct}=(b_1,...,b_d)$ and $D$ is a subset of the level product of $(T_1,...,T_d)$ satisfying
\[ |D\cap \big(T_1(m_n)\times ...\times T_d(m_n)\big)| \meg \ee |T_1(m_n)\times ...\times T_d(m_n)| \]
for every $n\mik N$, there exists a finite vector strong subtree $\bfcs$ of $\bfct$ of height $k$ such that the level product
of $\bfcs$ is a subset of $D$. The least integer $N$ with this property will be denoted by $\dhl(b_1,...,b_d|k,\ee,M)$.
\end{thm}
Notice, however, that the reduction of Theorem \ref{t51} to Theorem \ref{t12} via compactness is noneffective and gives no estimate
for the numbers $\dhl(b_1,...,b_d|k,\ee,M)$. The natural problem of getting explicit upper bounds for the ``density Halpern--L\"{a}uchli
numbers" is studied in \cite{DKT1}.


\end{document}